\documentclass{article}
\usepackage{amssymb}
\usepackage{amsmath}
\usepackage{amsthm}
\numberwithin{equation}{section}
\usepackage{hyperref}
\usepackage{pbox}
\usepackage{graphicx}
\usepackage[T1]{fontenc}
\usepackage{verbatim}

\usepackage{amsrefs}

\DefineSimpleKey{bib}{primaryclass}{}
\DefineSimpleKey{bib}{archiveprefix}{}

\BibSpec{arXiv}{%
  +{}{\PrintAuthors}{author}
  +{,}{ \textit}{title}
  +{}{ \parenthesize}{date}
  +{,}{ arXiv }{eprint}
  +{,}{ primary class }{primaryclass}
}

\usepackage{caption}
\usepackage{subcaption}
\usepackage{float}
\usepackage{pax}
\usepackage{pdfpages}

\usepackage{enumerate}
\usepackage[shortlabels]{enumitem}
\usepackage{afterpage}

\usepackage{tikz-cd} 

\usepackage{tikz}
\usetikzlibrary{decorations.pathreplacing,shapes.misc}
\usepackage[all]{xy}

\theoremstyle{plain}
\newtheorem{theorem}{Theorem}[section]

\theoremstyle{plain}
\newtheorem{definition}[theorem]{Definition}

\theoremstyle{plain}
\newtheorem{proposition}[theorem]{Proposition}

\theoremstyle{plain}
\newtheorem{lemma}[theorem]{Lemma}

\DeclareMathOperator{\Pf}{Pf}
\DeclareMathOperator{\Gr}{Gr}
\newcommand{\PP}{\mathbb P}

\title{Deformations of half-canonical Gorenstein curves in codimension four}
\author{Patience Ablett \and Stephen Coughlan}
\date{}

\begin{document}

\maketitle

\begin{abstract}
    Recent work of Ablett \cite{ablett2021halfcanonical} and Kapustka, Kapustka, Ranestad, Schenck, Stillman and Yuan \cite{kapustka2021quaternary} outlines a number of constructions for singular Gorenstein codimension four varieties. Earlier work of Coughlan, Go\l{}\c{e}biowski, Kapustka and Kapustka \cite{coughlan2016arithmetically} details a series of nonsingular Gorenstein codimension four constructions with different Betti tables. In this paper we exhibit a number of flat deformations between Gorenstein codimension four varieties in the same Hilbert scheme, realising many of the singular varieties as specialisations of the earlier nonsingular varieties.
\end{abstract}

\section{Introduction}

\subsection*{Gorenstein varieties in low codimension}
The explicit construction of Gorenstein varieties is well understood in codimension three or less. Reid \cite{reid2015gorenstein} gave a general structure theorem for codimension four, but constructing Gorenstein ideals of codimension four is still not fully understood. Thus the case of codimension four is still an area of active study. By understanding concrete examples of codimension four Gorenstein varieties and the relationships between these varieties, we can hope to understand more about the general case.

\begin{definition}
For a projective variety $X \subset \mathbb{P}^n$, we say $X$ is projectively Gorenstein, or alternatively arithmetically Gorenstein, if the coordinate ring $R$ of $X$ is Cohen--Macaulay and the canonical module $\omega_R \cong R(a)$, where $a$ represents a shift in the grading.
\end{definition}

We henceforth refer to such varieties as Gorenstein varieties. The condition on the canonical module leads to interesting symmetry in the free resolution for such varieties. In codimension two it is shown by Serre \cite{serre1960modules} that a variety is Gorenstein if and only if it is a complete intersection. In codimension three, the Buchsbaum--Eisenbud structure theorem \cite{10.2307/2373926} shows that all Gorenstein ideals of height 3 are given by the $2n \times 2n$ Pfaffians of a $(2n+1) \times (2n+1)$ skew-symmetric matrix, for $n \in \mathbb{Z}_{\geq 1}$.

\subsection*{Attempts to classify codimension 4 Gorenstein varieties}

Recent work on Gorenstein varieties has focused in particular on Gorenstein Calabi--Yau threefolds, henceforth GoCY threefolds. Studying GoCY threefolds in $\mathbb{P}^7$ is an interesting subcase of the Gorenstein codimension four problem. This was first studied by Bertin \cite{bertin2009examples}, and later by Coughlan, Go\l{}\c{e}biowski, Kapustka and Kapustka, who published a list of nonsingular GoCY threefolds in $\mathbb{P}^7$ \cite{coughlan2016arithmetically} and gave evidence to suggest this list may be complete. Schenck, Stillman and Yuan made further progress, listing every possible Betti table for Artinian Gorenstein algebras of codimension and regularity four \cite{schenck2020calabiyau}. The regularity four condition is a direct result of the Calabi--Yau condition. Note that since the coordinate ring $R$ for a Gorenstein variety $X$ is Cohen--Macaulay, cutting by a regular sequence leaves an Artinian ring with the same Betti table. Thus any higher dimensional codimension four Gorenstein variety whose coordinate ring has regularity four will have a Betti table given in \cite{schenck2020calabiyau}. For each Betti table of \cite{schenck2020calabiyau}, there are explicit descriptions of one or more families of varieties with that Betti table (see \cite{coughlan2016arithmetically}, \cite{schenck2020calabiyau}, \cite{ablett2021halfcanonical}, \cite{kapustka2021quaternary}). Further, the Betti tables of \cite{schenck2020calabiyau} may be split into two parts. Those appearing in Table 1 of \cite{schenck2020calabiyau} correspond to at least one family of nonsingular threefolds in $\mathbb{P}^7$ (see \cite{coughlan2016arithmetically}), and are therefore referred to as the CGKK Betti tables. Those appearing in Table 2 of \cite{schenck2020calabiyau} do not occur for nonsingular threefolds (see \cite{schenck2020calabiyau}), and we refer to these as the SSY Betti tables.
\renewcommand{\arraystretch}{1.3}
\begin{table}
\begin{center}
\begin{tabular}[c]{|l|l|l|}
\hline
 \multicolumn{3}{| c |}{Classifying curves by genus and degree}\\
 \hline
 \rule{0pt}{35pt}\pbox{2.5cm}{Degree\\} & \pbox{2.5cm}{Genus\\} & \pbox{2.5cm}{Corresponding \\ Betti table\\}\\
 \hline
 14 & 15 & CGKK 1 \\
 15 & 16 & CGKK 2, SSY 7, SSY 8 \\
 16 & 17 & CGKK 3, SSY 3, SSY 4, SSY 6 \\
 17 & 18 & CGKK 4, CGKK 5,6, SSY 2, SSY 5 \\
 18 & 19 & CGKK 7,8, SSY 1 \\
 19 & 20 & CGKK 9,10 \\
 20 & 21 & CGKK 11 \\
 \hline
\end{tabular}
\end{center}
\caption{A stratification of the Betti tables of \cite{schenck2020calabiyau}.}
\label{tab:table12}
\end{table}

\subsection*{The results of this paper}

It was shown in \cite{schenck2020calabiyau} that the SSY Betti tables cannot correspond to smooth GoCY threefolds. Moreover, almost all the families of curves given in \cite{ablett2021halfcanonical} are families of singular stable nodal curves, in contrast to the families of nonsingular threefolds of \cite{coughlan2016arithmetically}. See table \ref{tab:table12} for a description of the families of curves in terms of their genus and degree. It thus becomes an interesting question of whether the singular varieties in \cite{ablett2021halfcanonical}, \cite{kapustka2021quaternary} can be smoothed to one of the nonsingular varieties of \cite{coughlan2016arithmetically} in the same Hilbert scheme. We summarise our results in the following theorem, which answers this question in the affirmative. 
\begin{theorem}\label{bigtheorem}
Let $\beta$ be an SSY Betti table of \cite{schenck2020calabiyau}. Then
    \begin{enumerate}[\normalfont(i)]
        \item There is at least one family of curves with Betti table $\beta$. 
        \item Each family in part (i) contains a subfamily whose general element admits a flat deformation to a curve with a CGKK Betti table.
    \end{enumerate}
\end{theorem}
Refined statements listing all the families and subfamilies that we consider are in Figures \ref{deg15}, \ref{deg16}, \ref{fig:deg17} and \S\ref{sec5}.

By a result of Hartshorne \cite{hartshorne1966connectedness}, the Hilbert scheme parametrising subschemes of $\PP^n_k$ with fixed Hilbert function is connected. Hence Theorem \ref{bigtheorem} is unsurprising. On the other hand, our varieties and subsequent deformations are outlined concretely in terms of ideal generators, which does not follow from Hartshorne's result.

We can often express these generators using Pfaffians of matrices and other similar techniques, giving a simple description of both the ideal and the deformation. We repeatedly utilise the ``Cramer's rule'' format originally seen in \cite{kustin1980algebra} and explained in \cite[\S 2.8]{Papadakis2000Kustin--MillerComplexes}. Consider a $3 \times 4$ matrix $M$ and column vector $v$ of length 4 with entries in some polynomial ring, and a further parameter $s$. Then the equations obtained from \[Mv=0, \quad \bigwedge^3 M=sv\] define a variety. By $\bigwedge^3 M=sv$ we mean that $sv_i = (-1)^i\det M_{\widehat{i}}$, where $M_{\widehat{i}}$ is the $3 \times 3$ matrix obtained by deleting the $i$th column of $M$. Several of our families of varieties may be described by imposing certain conditions on $M$, $v$ and $s$. For example, when $M$ and $v$ have generic linear entries, and $s$ is also general and quadratic, we obtain a family of varieties with Betti table CGKK 4.

We explain our naming convention for the families described in \cite{coughlan2016arithmetically}, \cite{schenck2020calabiyau}, \cite{ablett2021halfcanonical} and \cite{kapustka2021quaternary}. For simplicity, in this paragraph we only consider families of curves. As observed in \cite{kapustka2021quaternary}, each Betti table is uniquely identified by its first row $b_{12}b_{23}b_{34}$. The symbol [$b_{12}b_{23}b_{34}$] refers to the locus $\mathcal{F}_{[b_{12}b_{23}b_{34}]}$ of quaternary quartics whose apolar ring has that Betti table. An additional letter inside the brackets (for example [300a]), specifies an irreducible subset (of the locus $\mathcal{F}_{[300]}$) as identified in \cite{kapustka2021quaternary}. We add decorations outside the brackets to denote the families discussed in this paper. For example, consider the family of curves [441b]a. The first row of the Betti table of a general curve in this family is 441. The b inside the brackets indicates that the family corresponds to a certain stratum of quartics, $\mathcal{F}_{[441b]}$. The a outside of the brackets is our primary identifier for this family of varieties.
Different families corresponding to the same stratum of quartics are differentiated by our primary identifier. For example, the two families corresponding to $\mathcal{F}_{[562]}$ are called [562]a and [562]b. A specialisation of a family is denoted by appending roman numerals to the name e.g.~[441b]a specialises to [441b]ai. We sometimes write [562]a to refer to the family of curves [562]a. An element of (the family of curves) [562]a is called a curve in [562]a or sometimes a curve of Type [562]a.

\begin{figure}[ht]
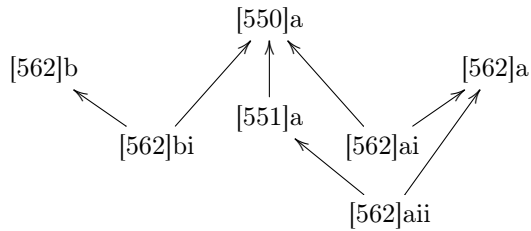

\[\xygraph{
!{<0cm,0cm>;<1cm,0cm>:<0cm,1.3cm>::}
!{(0,5)}*+{{\text{[550]a}}}="1"
!{(0,4)}*+{{\text{[551]a}}}="2"
!{(3,4.5)}*+{{\text{[562]a}}}="3"
!{(1.5,3.7)}*+{{\text{[562]ai}}}="4"
!{(1.6,3)}*+{{\text{[562]aii}}}="5"
!{(-3,4.5)}*+{{\text{[562]b}}}="6"
!{(-1.5,3.7)}*+{{\text{[562]bi}}}="7"
"2":"1"
"4":"1"
"7":"1"
"5":"2"
"4":"3"
"5":"3"
"7":"6"
}\]\caption{Degree 15 strata and incidences}
\label{deg15}
\end{figure}

Figure \ref{deg15} gives an overview of the results for degree 15. Similarly, figures \ref{deg16} and \ref{fig:deg17} describe the results for degrees 16 and 17 respectively. An arrow $A\to B$ means that we can exhibit a flat deformation whose central fibre is a general element of the stratum $A$ and whose general fibre is a general element of the stratum $B$.

It is interesting to compare and contrast our strata and their incidences with the stratification of the space of quartics in four variables in \cite{kapustka2021quaternary}. We do not claim to have a stratification of the Hilbert scheme of curves in $\PP^5$. Indeed, we know of other strata which are not included in the scope of this paper. Moreover, we have not counted parameters or moduli.

\subsection*{Lifting to higher dimensional varieties}
Using equations to describe our families of curves means that we can consider them more generally as families of codimension four varieties in a given $n$-dimensional projective space. These are referred to as liftings in \cite{kapustka2021quaternary}. In many situations, we can lift our smoothings of curves to smoothings of higher dimensional varieties too.
However, it may happen that the lifted variety does not have a smoothing, because the deformation is obstructed in some lower dimension.

We make this more precise in degree 15. Family [550]a lifts to dimension $6$, where it is the intersection of a cubic hypersurface with the $7$-dimensional projective cone over $\Gr(2,5)$. Any further lifting results in singularities where the cubic intersects the vertex of the cone.

Family [551]a lifts to at least dimension $5$ if we assume that $a_i,b_i$ are coordinates along with $x_i$ and $y$. Similarly, the special family [562]bi lifts to dimension $5$ if we assume that $L_0$, $L_1$, $L_2$, $L_3$ are coordinates. The smoothings of [551]a and [562]bi to [550]a both lift to dimension $5$ without obstruction.

On the other hand, family [562]a lifts to dimension $9$ if we assume that $a_{ij}$, $b_{ij}$ are coordinates along with $x_i,y_j$.
Since there is a deformation from [562]ai to [551]a and the maximal dimension of a lifting of [551]a is $6$, it follows that the smoothing of [562]ai must be obstructed in dimensions $\ge7$. Indeed, the deformation to [550]a collects all the terms involving $b_{ij}$ into the equation for the cubic hypersurface. The resulting $9$-dimensional variety of type [550]a is singular along the intersection of this cubic hypersurface with the $3$-dimensional vertex of the cone over $\Gr(2,5)$ with coordinates $b_{ij}$.

\subsection*{How we prove our results}

We prove Theorem \ref{bigtheorem} by constructing explicit flat families, typically over $\mathbb{A}^1$ but in one case over $\mathbb{A}^4$, whose special fibre is a singular variety with an SSY Betti table and whose general fibre is a variety in one of the families of Coughlan, Go\l{}\c{e}biowski, Kapustka and Kapustka. Explicit details of the deformations are outlined in sections \ref{sec2}, \ref{sec3}, \ref{sec4} and \ref{sec5}. Note that the deformations for the curves of Type~[551]a and [562]a are due to Jan Stevens.

To show that the deformations are indeed flat we utilise the fact that for a Noetherian integral scheme $T$ and a family $X \subset \mathbb{P}^n_T$, $X$ is flat over $T$ if the Hilbert polynomial is independent of $t \in T$ at every fibre $X_t$ (see \cite{MR0463157}, III.9, page 261). Thus in our case, where $t \in \mathbb{A}^1$ or occasionally $\mathbb{A}^4$, it is enough to check the dimension and corresponding Betti table of every fibre $X_t$, since the Hilbert polynomial can be directly calculated from this information. Table \ref{tab:table12} organises the Betti tables of \cite{schenck2020calabiyau} into strata according to the degree and genus of the stable curves in \cite{ablett2021halfcanonical}. Equidimensional varieties with Betti tables in the same strata will have the same Hilbert polynomial. We use Magma \cite{MR1484478} to check that our families have the correct dimension and Betti table.

In sections \ref{sec2}, \ref{sec3}, \ref{sec4} and \ref{sec5} we outline details of each family of varieties and its deformations. The different sections correspond to different degrees, with section \ref{sec2} describing the degree 15 varieties up to section \ref{sec5} describing the degree 18 varieties. Further each section is split into subsections, with each subsection describing a different construction in that degree and its possible deformations.

This work builds on the talk ``Gorenstein curves of codimension four'' given by Ablett at the MEGA 2022 conference, which discussed the results of \cite{ablett2021halfcanonical}. The original constructions in \cite{ablett2021halfcanonical} and the flat families in this paper make frequent use of the Magma \cite{MR1484478} computer algebra software. The code is available at
\begin{quote}\texttt{github.com/PatienceAblett/GorensteinCodim4Deformations}\end{quote}
and runs in the Magma online calculator.

Many of the varieties in the paper are defined using skew-symmetric matrices. We utilise the convention of only specifying the upper right triangular entries of the skew-symmetric matrix, since the rest of the entries are given by the skew-symmetry. We use ${\Pf_{\hat{i}}}M$ to denote the maximal Pfaffian of the submatrix of $M$ obtained by removing the $i$th row and column. When there is only one matrix used in a construction we drop the $M$ here for brevity.

\section{Deformations in degree 15}\label{sec2}
In this section we outline a series of deformations of degree 15 curves, exhibiting each of our singular curves in a flat family, whose general fibre is a nonsingular curve section of one of the constructions of Coughlan, Go\l{}\c{e}biowski, Kapustka and Kapustka. In our naming convention, family no.~2 of Coughlan, Go\l{}\c{e}biowski, Kapustka and Kapustka is called family [550]a. Note that we do not claim this is the only possible family of varieties with this Betti table. The family in \cite{ablett2021halfcanonical} with Betti table SSY 7 is called family [551]a. The two constructions with Betti table SSY 8, as seen in \cite{ablett2021halfcanonical} and \cite{kapustka2021quaternary}, are referred to as families [562]a and [562]b. Here the letter distinguishing between the two families is outside the brackets. Both families are associated to  the same irreducible locus of quaternary quartics from \cite{kapustka2021quaternary}. Note that the smoothing of a general curve in the Type~[551]a family is outlined in \cite{ablett2021halfcanonical}, as well as further details on all the families described in this section.

\subsection{Family~[550]a}
The family [550]a is originally outlined in \cite{coughlan2016arithmetically}, as $\Gr(2,5) \cap X_3 \cap \mathbb{P}^7$, where $X_3$ is a general cubic hypersurface. In other words, we take a linear $\mathbb{P}^7$ section of $\Gr(2,5)$ in its Pl\"{u}cker embedding and intersect with a cubic hypersurface. If we instead intersect $\Gr(2,5)$ with $\mathbb{P}^5$ to obtain a surface, this is the degree 5 del Pezzo surface. Subsequently intersecting with a cubic hypersurface defines a nonsingular curve with Betti table CGKK 2.
\begin{table}[h!]
    \begin{equation*}\begin{array}{c|c c c c c}
        & 0 & 1 & 2 & 3 & 4 \\ \hline
       0 & 1 & - & - & - & - \\
      1 & - & 5 & 5 & - &-\\
      2 &- & 1 & - & 1 & -\\
      3 &- & - & 5 & 5 & - \\
      4 & - & - & - & - & 1
    \end{array}
\end{equation*}
\caption*{Betti table CGKK 2~\cite{schenck2020calabiyau}}
\label{tab:table6}
\end{table}

\subsection{Family~[551]a}
We start by recasting family [551]a, which was originally constructed in  \cite{ablett2021halfcanonical}, as an example of the ``Cramer's rule'' format.

\begin{lemma}
Consider the following matrix $M$, vector $v$, and parameter $s$:
\begin{equation*}
    M = \begin{pmatrix}
    0 & a_0 & a_1 & a_2 \\
    0 & b_0 & b_1 & b_2 \\
    0 & c_0 & c_1 & c_2
    \end{pmatrix}, \quad
    v = \begin{pmatrix}
    F_2 \\
    x_0 \\
    x_1 \\
    x_2
    \end{pmatrix}, \quad 
    s=y.
\end{equation*} 
Here the $a_i$ and $b_i$ are linear forms, the $c_i$ are quadratic forms and $F_2$ is a cubic form in $x_0,x_1,x_2,x_3,x_4,y$. Let $C$ be the variety in $\PP^5$ defined by equations $Mv=0$, $\bigwedge^3 M=sv$. Then $C=C_1+C_2$ is a curve with Betti table SSY 7. We call this family [551]a.
\end{lemma}

\begin{table}[h!]
    \begin{equation*}\begin{array}{c|c c c c c}
        & 0 & 1 & 2 & 3 & 4 \\ \hline
       0 & 1 & - & - & - & - \\
      1 & - & 5 & 5 & 1 &-\\
      2 &- & 1 & 2 & 1 & -\\
      3 &- & 1 & 5 & 5 & - \\
      4 & - & - & - & - & 1
    \end{array}
\end{equation*}
\caption*{Betti table SSY 7~\cite{schenck2020calabiyau}}
\label{tab:table13}
\end{table}
We use the terminology $C_1 + C_2$ to denote the union of two curves which intersect transversely.

We describe the geometry of the Type~[551]a curve constructed in the above Lemma more precisely. Three of the five quadrics defining $C$ are 
\begin{equation*}
Q_0 = x_0y, \quad Q_1 = x_1y, \quad Q_2 = x_2y.
\end{equation*}
It follows that $C$ breaks into two pieces, $C_1 \subset \mathbb{P}^4_{\left<x_0,\dots,x_4\right>}$ and $C_2 \subset \mathbb{P}^2_{\left<x_3,x_4,y\right>}$. 

The three equations coming from the product $Mv=0$ are
\begin{equation*}
Q_3 = a_0x_0+a_1x_1+a_2x_2, \quad Q_4 = b_0x_0+b_1x_1+b_2x_2, \quad F_1 = c_0x_0+c_1x_1+c_2x_2.
\end{equation*}
Let $\Gamma$ be the $(2,2,3)$ complete intersection defined by $Q_3$, $Q_4$, $F_1$. Then $\Gamma$ contains the line $l = \mathbb{P}^1_{\left<x_3,x_4\right>}$ and 
$C_1$ is the residual curve to $l$ in $\Gamma$, which is defined by $Q_3,Q_4,F_1$ and a further quartic $H$, which is the determinant of the matrix
\begin{equation*}
    \begin{pmatrix}
    a_0 & a_1 & a_2 \\
    b_0 & b_1 & b_2 \\
    c_0 & c_1 & c_2
    \end{pmatrix}.
\end{equation*}
The plane curve $C_2$ is then defined by the more general quartic $H' = H + yF_2$, where $F_2$ is a cubic form. The ideal defining $C$ is $I_C = (Q_0,\dots,Q_4,F_1,H')$.

 Thus $C = C_1 + C_2 \subset \mathbb{P}^5$, where $C_1 \subset \mathbb{P}^4$ is residual to a line in a $(2,2,3)$ complete intersection and $C_2 \subset \mathbb{P}^2$ is a plane quartic ``ear'', meeting $C_1$ in four points on a line.
 
Jan Stevens showed that there is a flat family over $\mathbb{A}^1$ whose general fibre is a curve of Type~[550]a and whose fibre over $t=0$ is a general curve of Type~[551]a, see \cite{ablett2021halfcanonical}. Writing the equations using Cramer's rule allows us to describe this deformation to Type~[550]a with ease. Replacing $M$ with 
\begin{equation*}
    M_t = \begin{pmatrix}
    0 & a_0 & a_1 & a_2 \\
    0 & b_0 & b_1 & b_2 \\
    t & c_0 & c_1 & c_2
    \end{pmatrix},
\end{equation*}
the equations $M_tv=0$, $\bigwedge^3 M_t=sv$ now define a Type~[550]a variety if $t$ is invertible, as shown in \cite{ablett2021halfcanonical}.

\subsection{Family~[562]a}\label{2.8}
We present two different families of curves with Betti table SSY 8 and outline a deformation for each construction.
\begin{table}[h!]
    \begin{equation*}\begin{array}{c|c c c c c}
        & 0 & 1 & 2 & 3 & 4 \\ \hline
       0 & 1 & - & - & - & - \\
      1 & - & 5 & 6 & 2 &-\\
      2 &- & 2 & 4 & 2 & -\\
      3 &- & 2 & 6 & 5 & - \\
      4 & - & - & - & - & 1
    \end{array}
\end{equation*}
\caption*{Betti table SSY 8~\cite{schenck2020calabiyau}}
\label{tab:table15}
\end{table}

We construct the first family, which is called [562]a. Let $x_1,\dots,x_4,y_1,y_2$ be coordinates on $\PP^5$ and choose cubics 
\begin{align*}
    F_1 &= a_{13}x_3x_1 + a_{14}x_4x_1 + a_{23}x_3x_2 + a_{24}x_4x_2, \\
    F_2 &= b_{13}x_3x_1 + b_{14}x_4x_1 + b_{23}x_3x_2 + b_{24}x_4x_2
\end{align*}
where $a_{ij}$, $b_{ij}$ are linear forms. Now we can state:

\begin{lemma}
 Let $C$ be the curve in $\PP^5$ defined by the $4\times 4$ Pfaffians of the 
following two skew-symmetric matrices, and one additional equation:
\begin{equation}\label{[562]amats}
\renewcommand{\arraystretch}{1.3}
 M_1=\begin{pmatrix}
y_2 & 0 & \frac{\partial{F_2}}{\partial{x_1}} & \frac{\partial{F_2}}{\partial{x_2}} \\
& 0 & \frac{\partial{F_1}}{\partial{x_1}} & \frac{\partial{F_1}}{\partial{x_2}} \\
& & x_2 & -x_1 \\
&&& H_{34}
 \end{pmatrix},\
 M_2=\begin{pmatrix}
y_1 & 0 & \frac{\partial{F_2}}{\partial{x_3}} & \frac{\partial{F_2}}{\partial{x_4}} \\
& 0 & \frac{\partial{F_1}}{\partial{x_3}} & \frac{\partial{F_1}}{\partial{x_4}} \\
& & x_4 & -x_3 \\
&&& H_{12}
 \end{pmatrix},\
y_1y_2=0.
\end{equation}
Here $H_{12}$ and $H_{34}$ are general cubics.
Then $C$ has Betti table SSY 8. We call this family of curves [562]a.
\end{lemma}
Note that two of the Pfaffians of $M_1$ are also repeated as Pfaffians of $M_2$. These are
the cubics $F_1=\Pf_{\hat 1} M_1=\Pf_{\hat 1} M_2$, and 
$F_2=\Pf_{\hat 2} M_1=\Pf_{\hat 2} M_2$. Thus there are nine ideal generators.

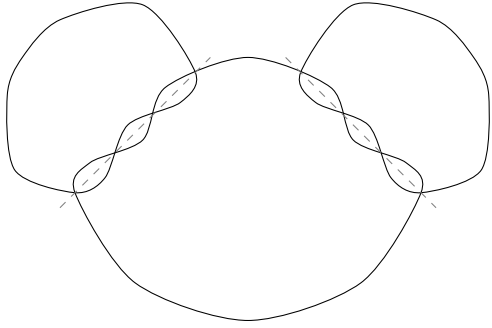
\begin{figure}[ht]
\begin{center}
\begin{tikzpicture}
\draw plot [smooth cycle] coordinates { 
(0.2,0.2) (0.5+0.1,0.5-0.1) (1-0.1,1+0.1) (1.5+0.1,1.5-0.1) (1.8,1.8)
(1.1,2.7) (0,2.5) (-0.5,2) (-0.7,1.5) (-0.6,0.5)};
\draw plot [smooth cycle] coordinates { 
(0.2,0.2) (0.5-0.1,0.5+0.1) (1+0.1,1-0.1) (1.5-0.1,1.5+0.1) 
(2.5,2)
(3.5+0.1,1.5+0.1) (4-0.1,1-0.1) (4.5+0.1,0.5+0.1) (4.8,0.2)
(4,-1) (2.5,-1.5) (1,-1)};
\draw plot [smooth cycle] coordinates { 
(4.8,0.2) (4.5-0.1,0.5-0.1) (4+0.1,1+0.1) (3.5-0.1,1.5-0.1) (3.2,1.8)
(3.9,2.7) (5,2.5) (5.5,2) (5.7,1.5) (5.6,0.5)};
\draw [gray, dashed] (0,0)--(2,2);
\draw [gray, dashed] (3,2)--(5,0);
\end{tikzpicture}
\end{center}
\caption{Curve [562]a -- Big Ears}
\label{bigears}
\end{figure}

Writing out the quadric equations from the Lemma, we get
\begin{equation*}
    Q_1=x_3y_1, \quad Q_2=x_4y_1, \quad Q_3=x_1y_2, \quad Q_4=x_2y_2, \quad Q_5=y_1y_2.
\end{equation*}
Hence $C=C_1+C_2+C_3$ where $C_1\subset\PP^3_{\left<x_1,\dots,x_4 \right>}$, $C_2 \subset \PP^2_{\left<x_3,x_4,y_2\right>}$, and $C_3 \subset \PP^2_{\left<y_1,x_1,x_2\right>}$.

First consider $C_1$. The two repeated Pfaffians $F_1$ and $F_2$ define a $(3,3)$ complete intersection which is residual to the lines $l_{34} = \PP^1_{\left<x_3,x_4\right>}$ and $l_{12} = \PP^1_{\left<x_1,x_2\right>}$.
The residual curve is $C_1$, with ideal $(F_1,F_2,q_{12},q_{34})$, where the quartics $q_{12}$ and $q_{34}$ are given by the determinants of Jacobian matrices as follows:
\begin{equation*}
q_{34}=\frac{\partial(F_1,F_2)}{\partial(x_1,x_2)} = 
    \begin{pmatrix}
    \frac{\partial F_1}{\partial x_1} & \frac{\partial F_1}{\partial x_2} \\
    \frac{\partial F_2}{\partial x_1} & \frac{\partial F_2}{\partial x_2}
    \end{pmatrix}, \quad q_{12} = \frac{\partial(F_1,F_2)}{\partial(x_3,x_4)}=
    \begin{pmatrix}
    \frac{\partial F_1}{\partial x_3} & \frac{\partial F_1}{\partial x_4} \\
    \frac{\partial F_2}{\partial x_3} & \frac{\partial F_2}{\partial x_4}
    \end{pmatrix}.
\end{equation*}
 The quartic $G_{34}=q_{34}+y_2H_{34}$ defines $C_2 \subset \PP^2_{\left<x_3,x_4,y_2\right>}$, and $G_{12}=q_{12}+y_1H_{12}$ defines $C_3 \subset \PP^2_{\left<y_1,x_1,x_2\right>}$. Here $H_{12}$ and $H_{34}$ are general cubics.
 
 Thus $C \subset \PP^5$ is a union of three curves, $C_1 \subset \PP^3$ residual to a line pair in a $(3,3)$ complete intersection, and two plane quartic ``ears'' $C_2$ and $C_3$. Each quartic ear meets $C_1$ in four nodes lying on a residual line. We therefore refer to a curve in family [562]a colloquially as ``Big Ears'' (see Figure \ref{bigears} for a schematic diagram).

\begin{proposition}
There is a specialisation [562]ai of the family [562]a such that a general curve in [562]ai admits a deformation to a curve in family [550]a. We construct a flat family over $\mathbb{A}^1$ whose special fibre is a curve in family [562]ai and whose general fibre is a curve in family [550]a.
\end{proposition}
\begin{proof}
The following deformation was first explained to us by Jan Stevens. Let [562]ai denote the special subfamily of family [562]a where $H=H_{12}=H_{23}$. Let $t \in \mathbb{A}^1$ be a degree 0 deformation parameter. We perturb the matrices \eqref{[562]amats} to the following:
\begin{equation*}\label{[562]amatsnew}
\renewcommand{\arraystretch}{1.3}
 M_1'=\begin{pmatrix}
y_2 & t & \frac{\partial{F_2}}{\partial{x_1}} & \frac{\partial{F_2}}{\partial{x_2}} \\
& 0 & \frac{\partial{F_1}}{\partial{x_1}} & \frac{\partial{F_1}}{\partial{x_2}} \\
& & x_2 & -x_1 \\
&&& H
 \end{pmatrix},\
 M_2'=\begin{pmatrix}
y_1 & t & \frac{\partial{F_2}}{\partial{x_3}} & \frac{\partial{F_2}}{\partial{x_4}} \\
& 0 & \frac{\partial{F_1}}{\partial{x_3}} & \frac{\partial{F_1}}{\partial{x_4}} \\
& & x_4 & -x_3 \\
&&& H
 \end{pmatrix}.
\end{equation*}
The quadric $Q_5$ also becomes $Q_5' = y_1y_2 - t^2(a_{14}a_{23}-a_{13}a_{24})$. Since we assumed that $H_{12}=H_{34}$, we maintain the coincidences $\Pf_{\hat i} M'_1=\Pf_{\hat i} M'_2$ for $i=1,2$.
We write out the perturbed Pfaffians in full. The Pfaffian quadrics deform to
\begin{align*}
 Q'_3&=\Pf_{\hat 4} M'_1=x_1y_2+t\tfrac{\partial{F_1}}{\partial{x_2}},&
 Q'_4&=\Pf_{\hat 5} M'_1=x_2y_2-t\tfrac{\partial{F_1}}{\partial{x_1}},\\
 Q'_1&=\Pf_{\hat 4} M'_2=x_3y_1+t\tfrac{\partial{F_1}}{\partial{x_4}},&
 Q'_2&=\Pf_{\hat 5} M'_2=x_4y_1-t\tfrac{\partial{F_1}}{\partial{x_3}},
\end{align*}
one of the repeated cubic Pfaffians deforms to
\[
 F_2'=\Pf_{\hat 2}M_1'=\Pf_{\hat 2}M_2'=F_2+tH,
\]
and the three remaining Pfaffians (one repeated) are unchanged:
\begin{align*}
 F_1'&=\Pf_{\hat 1} M'_1=\Pf_{\hat 1} M'_2=F_1,\\
 G_{12}'&=\Pf_{\hat 3} M'_1=y_1H+\det\tfrac{\partial(F_1,F_2)}{\partial(x_3,x_4)}=G_{12}, \\
 G_{34}'&=\Pf_{\hat 3} M'_2=y_2H+\det\tfrac{\partial(F_1,F_2)}{\partial(x_1,x_2)}=G_{34} .
\end{align*}
If $t$ is invertible, then these three unchanged Pfaffians are redundant. That is, they are not needed as ideal generators. Indeed, using $Q_i'$ to eliminate the partial derivatives of $F_1$, we have
\[
tF_1=t\tfrac{\partial{F_1}}{\partial{x_1}}x_1+t\tfrac{\partial{F_1}}{\partial{x_2}}x_2
=(x_2y_2)x_1+(-x_1y_2)x_2=0,
\]
and 
\begin{align*}
tG_{12}&=ty_1H+t\det\begin{pmatrix}
\tfrac{\partial{F_1}}{\partial{x_3}} &
\tfrac{\partial{F_1}}{\partial{x_4}} \\
\tfrac{\partial{F_2}}{\partial{x_3}} &
\tfrac{\partial{F_2}}{\partial{x_4}}
\end{pmatrix}=
ty_1H+\det\begin{pmatrix}
x_4y_1 & -x_3y_1 \\
\tfrac{\partial{F_2}}{\partial{x_3}} &
\tfrac{\partial{F_2}}{\partial{x_4}}
\end{pmatrix},\\
&=ty_1H+y_1(\tfrac{\partial{F_2}}{\partial{x_3}}x_3+\tfrac{\partial{F_2}}{\partial{x_4}}x_4)=y_1F'_2.
\end{align*}
Similarly, $G_{34}'$ reduces to $y_2F_2'$ and thus, $F_1'$, $G_{12}'$ and $G_{34}'$ are redundant when $t$ is invertible. The remaining Pfaffian equations $Q_1'$, $Q_2'$, $Q_3'$, $Q_4'$, $F_2'$ and the extra equation  $y_1y_2 - t^2(a_{14}a_{23}-a_{13}a_{24})$ define a curve in family [550]a as follows:
\[
\Pf\begin{pmatrix}
& y_1 & x_2 & ta_{14} & ta_{13} \\
& & -x_1 & ta_{24} & ta_{23} \\
& & & -x_3 & x_4 \\
& & & & y_2
\end{pmatrix}, \ F_2+tH.
\]
Consider the ideal $I_t=(Q_1',\dots,Q_4',Q_5,F_1,F_2',G_{12},G_{34})$. By the above, when $t$ is invertible $I_t$ defines a curve of Type~[550]a. Clearly $I_0$ defines a curve of Type~[562]ai. Every fibre has the same dimension and degree, and consequently the same Hilbert polynomial. Since the Hilbert polynomial at each fibre is independent of $t$ this is indeed a flat family over $\mathbb{A}^1$.
\end{proof}

There is a second, symmetric deformation to [550]a if we put $t$ in entry $23$ of the perturbed matrices $M'_i$, instead of entry $13$.

\begin{proposition}
Let [562]aii be the specialisation of the family of curves [562]a obtained by requiring $H_{34}=0$. Then a general curve in [562]aii can be deformed to a curve in [551]a. We exhibit a flat family over $\mathbb{A}^1$ whose fibre at $t=0$ is a general curve in [562]aii, and whose general fibre is a curve in [551]a.
\end{proposition}
\begin{proof}

The following deformation was first explained to us by Jan Stevens. Let [562]aii denote the special subfamily of [562]a where $H_{34}=0$. Replace $M_1$ with the matrix
\begin{equation*}
    M_1'=
    \begin{pmatrix}
        y_2 & t & \frac{\partial{F_2}}{\partial{x_1}} & \frac{\partial{F_2}}{\partial{x_2}} \\
        & 0 & \frac{\partial{F_1}}{\partial{x_1}} & \frac{\partial{F_1}}{\partial{x_2}} \\
        & & x_2 & -x_1 \\
        &&& 0
    \end{pmatrix},
\end{equation*}
and leave $M_2$ as before. Two Pfaffians are perturbed and the rest remain unchanged. Specifically, we now have 
\begin{equation*}
 Q'_3=\Pf_{\hat 4} M'_1=x_1y_2+t\tfrac{\partial{F_1}}{\partial{x_2}},\ \ \
 Q'_4=\Pf_{\hat 5} M'_1=x_2y_2-t\tfrac{\partial{F_1}}{\partial{x_1}},
\end{equation*}
and we can use these to show that $G_{01}$ and $F_1$ are not needed as ideal generators when $t$ is invertible:
\begin{align*}
tF_1&=t\tfrac{\partial{F_1}}{\partial{x_1}}x_1+t\tfrac{\partial{F_1}}{\partial{x_2}}x_2=(x_2y_2)x_1+(-x_1y_2)x_2=0, \\
tG_{34}&=t\det\tfrac{\partial(F_1,F_2)}{\partial(x_1,x_2)}=(x_2y_2)\tfrac{\partial{F_2}}{\partial{x_2}}-(-x_1y_2)\tfrac{\partial{F_2}}{\partial{x_1}}=y_2F_2.
\end{align*}

Moreover, for $t$ invertible, the remaining seven generators define a curve in [551]a as follows:
\[\renewcommand{\arraystretch}{1.3}
M=\begin{pmatrix}
0&a_{23}&a_{24}&x_1\\
0&a_{13}&a_{14}&-x_2\\
0&\frac{\partial{F_2}}{\partial{x_3}} & \frac{\partial{F_2}}{\partial{x_4}}&0   
  \end{pmatrix},\quad
v=\begin{pmatrix}
H_{12}\\tx_3\\tx_4\\y_2
\end{pmatrix},\quad
s=y_1.
\]
Consequently the ideal $I_t=(Q_1,Q_2,Q_3',Q_4',Q_5,F_1,F_2,G_{12},G_{23})$ defines a Type~[562]aii curve when $t=0$ and a Type~[551]a curve otherwise, and this is again a flat family over $\mathbb{A}^1$.
\end{proof}

The specialisation $H_{34}=0$ imposes a nodal ear on the general curve of Type~[562]aii. The deformation smooths the central curve and the nodal ear together into a curve of degree 11 which is residual to a line in a $(2,2,3)$ complete intersection, just as in Type~[551]a. The other ear is deformed trivially.

\subsection{Family~[562]b}
The second family with Betti table SSY 8 is outlined in \cite{kapustka2021quaternary}, and is similar to family [441b]a described in \ref{[441b]}. 
\begin{lemma}
 Let $x_0,x_1,y_0,y_1,z_0,z_1$ be coordinates on $\PP^5$ and let $C$ be a curve in $\PP^5$ defined by the following nine equations:
 \begin{gather*}
Q_1=x_0y_0, \quad Q_2 = x_0y_1, \quad Q_3 = x_1y_0, \quad Q_4=x_1y_1,\\
F_1 = A_0y_0-A_1y_1, \quad F_2 = B_0y_0-B_1y_1, \\
G_1 = L_0x_0-L_1x_1, \quad G_2=D_0x_0-D_1x_1,\quad H.
\end{gather*}
Here $A_0,A_1,B_0,B_1$ are quadratic forms, $L_0,L_1$ are linear forms, $D_0,D_1$ are cubic forms and $H$ is a quartic which agrees with $A_0B_1-B_0A_1$ on $\mathbb{P}^3_{\left<y_0,y_1,z_0,z_1\right>}$ and which agrees with $L_0D_1-L_1D_0$ on $\mathbb{P}^3_{\left<x_0,x_1,z_0,z_1\right>}$.
Then $C$ has Betti table SSY 8. We call this family of curves [562]b.
\end{lemma}

The four quadrics $Q_1,Q_2,Q_3,Q_4$ are reducible so $C=C_1 + C_2 \subset \mathbb{P}^5$ where $C_1 \subset \mathbb{P}^3_{\left<y_0,y_1,z_0,z_1\right>}$, $C_2 \subset \mathbb{P}^3_{\left<x_0,x_1,z_0,z_1\right>}$.
Moreover, $C_1$ is residual to the line $y_0=y_1=0$ in the complete intersection defined by $(F_1,F_2)$. Similarly $C_2$ is residual to the line $x_0=x_1=0$ in the complete intersection defined by $(G_1,G_2)$. The curves $C_1$ and $C_2$ are each defined by an additional determinantal quartic, given by $A_0B_1-B_0A_1$ for $C_1$ and $L_0D_1-L_1D_0$ for $C_2$. The condition on $H$ implies that these quartics agree with the restriction of $H$ to the line $\mathbb{P}^1_{\left<z_0,z_1\right>}$. 

This determinantal condition on $H$ is tricky, because the entries of the two $2\times2$ matrices have different degrees: 
\begin{equation*}
    \begin{pmatrix}
    3 & 1 \\
    3 & 1
    \end{pmatrix} \quad \text{and} \quad
    \begin{pmatrix}
    2 & 2 \\
    2 & 2
    \end{pmatrix}.
\end{equation*}
To show that family [562]b exists, we focus on the case that the cubics $D_0$ and $D_1$ break up into a linear and quadric form or three linear forms.
Let [562]bi denote the special case
\begin{gather*}
    L_1=x_1,\quad A_0=L_3x_1,\quad B_0=L_0L_1,\quad B_1=L_0L_2,\\ D_0=L_0L_2L_3, \quad  D_1=L_1A_1
\end{gather*}
where $L_0,L_1,L_2,L_3$ are general linear forms. Then
 \begin{gather*}
        F_1=L_3x_1y_0-A_1y_1, \quad F_2=L_0L_1y_0-L_0L_2y_1, \\
        G_1 = L_0x_0-x_1^2, \quad G_2 =L_0L_2L_3x_0-L_1A_1x_1, \quad
        H=L_0L_2L_3x_1 - L_0L_1A_1.
    \end{gather*}
Together with the quadrics $Q_i$, these define a reduced curve with several irreducible components.

\begin{proposition}
The general curve in the specialised family [562]bi as described above can be deformed to a curve of Type~[550]a. We exhibit a flat family over $\mathbb{A}^1$ whose 
special fibre at $t=0$ is of Type~[562]bi and whose general fibre is of Type~[550]a.
\end{proposition}
\begin{proof}
Let $t$ in $\mathbb{A}^1$ be a degree 0 deformation parameter and consider the deformed quadrics
\begin{align*}
        Q_1'&=x_0y_0-tL_2x_1, &
        Q_2'&=x_0y_1-tL_1x_1, \\
        Q_3'&=x_1y_0-tL_2L_0, &
        Q_4'&=x_1y_1-tL_1L_0.
\end{align*}

When $t$ is invertible the ideal of the five quadrics $(Q_1',Q_2',Q_3',Q_4',G_1)$ can be generated by the $4 \times 4$ Pfaffians of the skew-symmetric matrix
\begin{equation*}
    M = \begin{pmatrix}
     & x_0 & x_1 & L_2 & 0 \\
     & & 0 & tL_1 & tx_1 \\
     & & & y_1 & tL_0 \\
     & & & & y_0 \\
    \end{pmatrix}.
\end{equation*}
Consequently when $t$ is invertible $(Q_1',Q_2',Q_3',Q_4',G_1)$ defines a degree 5 del Pezzo surface. Note that this is untrue when $t=0$, since it is necessary to be able to cancel $t$. Furthermore, we can show that when $t$ is invertible, $F_2,G_2$ and $H$ are not needed as ideal generators. Indeed, we have
\begin{align*}
tF_2 &= y_1Q_3' - y_0Q_4', \\
tG_2 &= A_1Q_2' + x_0F_1 - L_3x_1Q_1' + tL_2L_3G_1, \\
tH &= A_1Q_4' - L_3x_1Q_3' + x_1F_1.
\end{align*}
It follows that $F_1$ defines a cubic hypersurface in the del Pezzo surface cut out by the five quadrics. The ideals $I_t = (Q_1',Q_2',Q_3',Q_4',G_1,G_2,F_1,F_2,H)$ thus define a flat family of curves $C_t$, where $C_0$ is in family [562]bi and $C_t$ is in [550]a for $t$ invertible. Since the Hilbert polynomial is independent of $t$, this is indeed a flat family.
\end{proof}

\section{Deformations in degree 16}\label{sec3}
We now turn our attention to the degree 16 varieties. There are two different irreducible subsets of the locus $\mathcal{F}_{[441]}$ corresponding to Betti table SSY 6, denoted $\mathcal{F}_{[441a]}$ and $\mathcal{F}_{[441b]}$. The $a$ and $b$ are included inside the bracket to indicate the different components of the space of quarternary quartics in \cite{kapustka2021quaternary}. The corresponding families of varieties are denoted by $[441a]a$ and $[441b]a$, with both admitting specialisations which smooth to a nonsingular variety of \cite{coughlan2016arithmetically}. There are further families with Betti tables SSY 3 and SSY 4, which are denoted by [430]a and [420]a respectively. We connect each of these families to family [400]a, which is the family of complete intersections of four quadrics. The full results for this section are seen in figure \ref{deg16}.

\begin{figure}[ht]
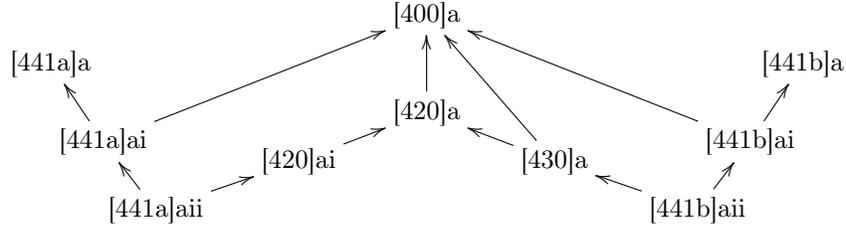

\[\xygraph{
!{<0cm,0cm>;<1cm,0cm>:<0cm,1.3cm>::}
!{(0,5)}*+{{\text{[400]a}}}="1"
!{(0,4)}*+{{\text{[420]a}}}="2"
!{(-1.7,3.5)}*+{{\text{[420]ai}}}="10"
!{(1.7,3.5)}*+{{\text{[430]a}}}="3"
!{(5,4.5)}*+{{\text{[441b]a}}}="4"
!{(4.3,3.7)}*+{{\text{[441b]ai}}}="5"
!{(3.6,3)}*+{{\text{[441b]aii}}}="6"
!{(-5,4.5)}*+{{\text{[441a]a}}}="7"
!{(-4.3,3.7)}*+{{\text{[441a]ai}}}="8"
!{(-3.6,3)}*+{{\text{[441a]aii}}}="9"
"2":"1"
"3":"1"
"8":"1"
"3":"2"
"10":"2"
"6":"3"
"5":"1"
"5":"4"
"6":"5"
"8":"7"
"9":"8"
"9":"10"
}\]\caption{Degree 16 strata and incidences}
\label{deg16}
\end{figure}

\subsection{Family~[400]a}
The family~[400]a consists of $(2,2,2,2)$ complete intersections. This family has Betti table CGKK 3. We construct a number of smoothings of varieties in the [420]a, [430]a, [441a]a and [441b]a families to varieties of this type.
\begin{table}[h!]
    \begin{equation*}\begin{array}{c|c c c c c}
        & 0 & 1 & 2 & 3 & 4 \\ \hline
       0 & 1 & - & - & - & - \\
      1 & - & 4 & - & - &-\\
      2 &- & - & 6 & - & -\\
      3 &- & - & - & 4 & - \\
      4 & - & - & - & - & 1
    \end{array}
\end{equation*}
\caption*{Betti table CGKK 3 ~\cite{schenck2020calabiyau}}
\label{tab:table4}
\end{table}

\subsection{Family~[420]a}
The original construction of the [420]a family is found in \cite{schenck2020calabiyau}. This family corresponds to Betti table SSY 4, and unlike the other ``SSY'' constructions is a family of nonsingular, irreducible curves. Note that when embedded as a threefold in $\mathbb{P}^7$, Schenck, Stillman and Yuan show that this threefold is singular.

\begin{table}[h!]
    \begin{equation*}\begin{array}{c|c c c c c}
        & 0 & 1 & 2 & 3 & 4 \\ \hline
       0 & 1 & - & - & - & - \\
      1 & - & 4 & 2 & - &-\\
      2 &- & 2 & 6 & 2 & -\\
      3 &- & - & 2 & 4 & - \\
      4 & - & - & - & - & 1
    \end{array}
\end{equation*}
\caption*{Betti table SSY 4~\cite{schenck2020calabiyau}}
\label{tab:table5}
\end{table}

Let $x_0,\dots,x_5$ be coordinates on $\mathbb{P}^5$ and consider the following matrix
\begin{equation*}
M = 
    \begin{pmatrix}
    & x_0 & x_1 & x_2 & 0 \\
    & & q_1 & q_2 & x_3 \\
    & & & q_3 & x_4 \\
    & & & & x_5 \\
    \end{pmatrix},
\end{equation*}
where the $q_i$ are general quadric forms. The $4 \times 4$ Pfaffians of this matrix define a codimension three variety, which when intersected with a fourth general quadric $Q$ is a Gorenstein codimension four curve of Type~[420]a.  The Pfaffians are given explicitly as
\begin{align*}
    {\text{Pf}_{\hat{1}}} &= x_5q_1 - x_4q_2 + x_3q_3, \\
    {\text{Pf}_{\hat{2}}} &= x_1x_5 - x_2x_4, \\
    {\text{Pf}_{\hat{3}}} &= x_0x_5 - x_2x_3, \\
    {\text{Pf}_{\hat{4}}} &= x_0x_4 - x_1x_3, \\
    {\text{Pf}_{\hat{5}}} &= x_2q_1 - x_1q_2 + x_0q_3.
\end{align*}
\begin{proposition}
A curve of Type~[420]a may be smoothed to the complete intersection of four quadrics. Explicitly, we construct a flat family of curves over $\mathbb{A}^1$ where the special fibre is a curve of the [420]a family and the general fibre is in the [400]a family.
\end{proposition}
\begin{proof}
The deformation is fairly simple and proceeds as follows. We introduce the deformation parameter $t \in \mathbb{A}^1$. Replacing the zero entry of the above matrix with the parameter $t$ we observe that ${\text{Pf}_{\hat{2}}},{\text{Pf}_{\hat{3}}}$ and ${\text{Pf}_{\hat{4}}}$ each gain a term, becoming
\begin{equation*}
{\text{Pf}_{\hat{2}}}' = x_1x_5 - x_2x_4 + tq_2, \quad {\text{Pf}_{\hat{3}}}' = x_0x_5 - x_2x_3 + tq_3, \quad {\text{Pf}_{\hat{4}}}' = x_0x_4 - x_1x_3 + tq_1.
\end{equation*}
From the syzygies between the three original quadric Pfaffians we obtain the following relations on the cubic Pfaffians:
\begin{align*}
    t{\text{Pf}_{\hat{1}}} &= x_0{\text{Pf}_{\hat{3}}}' - x_1{\text{Pf}_{\hat{2}}}' + x_2{\text{Pf}_{\hat{4}}}', \\
    t{\text{Pf}_{\hat{5}}} &= x_3{\text{Pf}_{\hat{3}}}' - x_4{\text{Pf}_{\hat{2}}}' + x_5{\text{Pf}_{\hat{4}}}'.
\end{align*}

It follows that when $t$ is invertible the cubic Pfaffians are not needed as ideal generators, and the ideal $({\text{Pf}_{\hat{2}}}',{\text{Pf}_{\hat{3}}}',{\text{Pf}_{\hat{4}}}',Q)$ defines a $(2,2,2,2)$ complete intersection. Again since the Hilbert polynomial is independent of $t$, this is a flat family over $\mathbb{A}^1$.
\end{proof}

\subsection{Family~[430]a}
We now describe family [430]a, again utilising the ``Cramer's rule'' format. 
\begin{table}[h!]
    \begin{equation*}\begin{array}{c|c c c c c}
        & 0 & 1 & 2 & 3 & 4 \\ \hline
       0 & 1 & - & - & - & - \\
      1 & - & 4 & 3 & - &-\\
      2 &- & 3 & 6 & 3 & -\\
      3 &- & - & 3 & 4 & - \\
      4 & - & - & - & - & 1
    \end{array}
\end{equation*}
\caption*{Betti table SSY 3~\cite{schenck2020calabiyau}}
\label{tab:table16}
\end{table}
Let $x_0,x_1,x_2,y_1,y_2,z$ be coordinates on $\mathbb{P}^5$, let $A,C$ be linear forms and $B,D,P,Q$ quadratic forms.
\begin{lemma}
    Consider the following matrix $M$, vector $v$ and parameter $s$:
    \begin{equation*}
M = 
    \begin{pmatrix}
    Q & P & A & C \\
    -x_0 & x_1 & 0 & 0 \\
    x_1 & -x_2 & 0 & 0 \\
    \end{pmatrix}, \quad
v = 
    \begin{pmatrix}
    y_2 \\
    y_1 \\
    D \\
    -B \\
    \end{pmatrix}, \quad s=z.
\end{equation*}
Then the seven equations obtained from $Mv=0, \quad \bigwedge^3 M=sv$ define a curve $C=C_1 + C_2 \subset \mathbb{P}^5$ with Betti table SSY 3. This is the [430]a family.
\end{lemma}
We describe the geometry of $C$ in more detail. Two of the four quadric ideal generators, $y_1z$ and $y_2z$ are reducible. Hence $C=C_1+C_2$, with $C_1 \subset \mathbb{P}^3_{\left<x_0,x_1,x_2,z\right>}$ and $C_2 \subset \mathbb{P}^4_{\left<x_0,x_1,x_2,y_1,y_2\right>}$. Consider the cubic surface scroll in $\PP^4_{\left<x_0,x_1,x_2,y_1,y_2\right>}$ defined by the vanishing of the $2 \times 2$ minors of the matrix
\begin{equation*}
    \begin{pmatrix}
    x_0 & x_1 & y_1 \\
    x_1 & x_2 & y_2 \\
    \end{pmatrix}.
\end{equation*}
The first two minors of the above matrix give the remaining two quadrics in the ideal of $C$. Consider the third minor $x_0x_2-x_1^2$. Let
\begin{align*}
    F &= A(x_0x_2-x_1^2) - Bz, \\
    G &= C(x_0x_2-x_1^2) - Dz.
\end{align*}
Then $C_1$ is the curve residual to $x_0x_2-x_1^2$ in the complete intersection $\Gamma \subset \PP^3_{\left<x_0,x_1,x_2,z\right>}$ defined by the vanishing of $(F,G)$. The cubic $AD-BC+Py_1+Qy_2$, cuts out $C_2$ inside the cubic scroll in $\PP^4$.

\begin{proposition}
A general curve of Type [430]a can be deformed to a curve of Type~[420]a, and to the complete intersection of four quadrics.
\end{proposition}
\begin{proof}
We consider the flat family $C_t$ over $\mathbb{A}^4$ with coordinates $t=(t_1,t_2,t_3,t_4)$. Denote by $N$ the $2 \times 2$ block of zeroes in $M$, which we replace with the $2 \times 2$ matrix
\begin{equation*}
N_t = 
    \begin{pmatrix}
    t_1 & t_2 \\
    t_3 & t_4 \\
    \end{pmatrix}.
\end{equation*}
Then $C_t$ is defined using the ``Cramer's rule'' format, with $N_t$ replacing $N$ in $M$. Depending on the rank of $N_t$, we observe three different possibilities for the homological behaviour of $C_t$. When $N_t$ has rank $0$, i.e. $N_t=N$, $C_t$ is a curve of type [430]a. When $N_t$ has rank 1, $C_t$ is a curve of type [420]a. When $N_t$ has rank 2, $C_t$ is a curve of type [400]a.

When $N_t$ has rank 1 we may write
\begin{equation*}
N_t = 
    \begin{pmatrix}
    t_1 & 0 \\
    0 & 0 \\
    \end{pmatrix},
\end{equation*}
where $t_1$ is invertible. Consequently the polynomials defining $C_t$ are given by
\begin{align*}
        Q_1 &= \tfrac{1}{t_1}zy_1 - Cx_1, \\ Q_2 &= \tfrac{1}{t_1}zy_2 - Cx_2, \\ Q_3 &= x_0y_2 - x_1y_1 - t_1D, \\ Q_4 &= x_1y_2 - x_2y_1, \\  F_1 &= A(x_0x_2 - x_1^2) - Bz + t_1(Qx_2 + Px_1), \\ F_2 &= C(x_0x_2 - x_1^2) - Dz, \\  F_3 &= AD - BC + Py_1 + Qy_2.
\end{align*}
Note that we can write $\frac{1}{t_1}$ since $t_1$ is invertible. We can rearrange $Q_1=0$ and $Q_2=0$ to obtain 
\begin{equation*}
Cx_1 = \tfrac{1}{t_1}zy_1, \quad Cx_2 = \tfrac{1}{t_1}zy_2.
\end{equation*}
Substituting these into $F_2$ we observe that $F_2 = \tfrac{z}{t_1}Q_3$, thus this polynomial is not required as an ideal generator. We show that the remaining polynomials define a curve of Type~[420]a. Five of the polynomials are given (modulo $Q_3$) by the $4 \times 4$ Pfaffians of the matrix
\begin{equation*}
    \begin{pmatrix}
     & z & t_1C & -t_1Q-x_0A & t_1P-x_1A \\
     & & 0 & x_1 & x_2 \\
     & & & y_1 & y_2 \\
     & & & & -B \\
    \end{pmatrix},
\end{equation*}
and the remaining polynomial $Q_3$ is the general quadric.

Finally if $N_t$ has rank $2$, then we assume the matrix is diagonal with 
\begin{equation*}
N_t = 
    \begin{pmatrix}
    t_1 & 0 \\
    0 & t_2 \\
    \end{pmatrix}.
\end{equation*}
The equations become

\begin{gather*}
Q_1 = zy_1 + Qt_1t_2 + Ax_0t_2 - Cx_1t_1, \quad
Q_2 = zy_2 - Pt_1t_2 + Ax_1t_2 - Cx_2t_1, \\
Q_3 = x_0y_2 - x_1y_1 - t_1D, \quad
Q_4 = x_1y_2 - x_2y_1 - t_2B, \\
F_1 = A(x_0x_2 - x_1^2) - Bz + t_1(Qx_2 + Px_1), \\
F_2 = C(x_0x_2 - x_1^2) - Dz + t_2(Qx_1 + Px_0), \\
F_3 = AD - BC + Py_1 + Qy_2.
\end{gather*}

We can eliminate $P$ and $Q$, since 
\begin{align*}
    Q &= \tfrac{1}{t_1t_2}(-zy_1 - Ax_0t_2 + Cx_1t_1), \\
    P &= \tfrac{1}{t_1t_2}(zy_2 + Ax_1t_2 - Cx_2t_1),
\end{align*}
and use this to show $F_1, F_2$ and $F_3$ are not necessary to generate the ideal. Thus when $N$ has rank two $C_t$ is a curve of Type~[400]a. Since the Hilbert polynomial at each fibre is independent of $t$ we have a flat family of curves over $\mathbb{A}^4$.
\end{proof}

\subsection{Family~[441a]a}
In \cite{ablett2021halfcanonical} we present two different families with Betti table SSY 6.
\begin{table}[h!]
    \begin{equation*}\begin{array}{c|c c c c c}
        & 0 & 1 & 2 & 3 & 4 \\ \hline
       0 & 1 & - & - & - & - \\
      1 & - & 4 & 4 & 1 &-\\
      2 &- & 4 & 8 & 4 & -\\
      3 &- & 1 & 4 & 4 & - \\
      4 & - & - & - & - & 1
    \end{array}
\end{equation*}
\caption*{Betti table SSY 6~\cite{schenck2020calabiyau}}
\label{tab:table7}
\end{table}

We outline a deformation for each family, starting with [441a]a.
\begin{lemma}
Let $C \subset \mathbb{P}^4_{\left<x_0,\dots,x_4,y\right>}$ be a curve defined by 
the quadrics 
\[Q_0=x_0y, \quad Q_1=x_1y, \quad Q_2=x_2y,\] the quartic 
\[q= (d_0d+d_1b+d_2c)x_3 - c^2 + b^2+Py,\] where $P$ is a cubic form in $(x_3,x_4,y)$ and  $d_0,d_1,d_2$ are linear forms, and
the five Pfaffians of
    \begin{equation}\label{mat}
    M=\begin{pmatrix}
    & b & c & d & 0 \\
    & & x_3 & x_1 & x_0 \\
    & & & x_0 & x_1 \\
    & & & & x_2 \\
    \end{pmatrix},
    \end{equation}
    where $b,c$ and $d$ are quadratic forms, with $d=d_0x_0+d_1x_1+d_2x_2$. Then $C$ has Betti table SSY 6. We denote this family of curves by [441a]a.
\end{lemma}
Brown, Kerber and Reid \cite{brown2012fano} discuss conditions that can be imposed on the entries of a matrix so that its Pfaffians lie in a given ideal. Two of the known solutions to this problem are called Tom and Jerry. The matrix $M$ is in the Jerry format, which appears frequently in this paper. Thus the Pfaffians of $M$ lie in the ideal $(x_0,x_1,x_2)$. It is an open problem as to whether we can construct similar deformations using the Tom matrix format instead.

We discuss the geometry of $C$. Let $\Gamma \subset \mathbb{P}^4_{\left<x_0,\dots,x_4\right>}$ be the codimension three variety cut out by the Pfaffians of $M$, and let $l$ be the line defined by $(x_0,x_1,x_2)$ in this copy of $\mathbb{P}^4$. Note that the Pfaffians of $M$ and the quartic $q$ are a set of minimal ideal generators for the colon ideal $(I_{\Gamma} \colon I_l)$. By looking at the quadrics defining $C$, we observe that $C=C_1+C_2$, where $C_1 \subset \mathbb{P}^4_{\left<x_0,\dots,x_4\right>}$ is the curve residual to $l$ in $\Gamma$, defined by the five Pfaffians of $M$ and the quartic $q$. On the other hand, $C_2 \subset \mathbb{P}^2_{\left<x_0,x_1,x_2\right>}$ is the plane quartic defined by $q$.
 The Pfaffians of $M$ are
    \begin{align*}
    {\Pf_{\hat{1}}} &= x_0^2 - x_1^2 + x_2x_3,\\
    {\Pf_{\hat{2}}} &= x_2c - x_1d,\\
    {\Pf_{\hat{3}}} &= x_2b-x_0d,\\
    {\Pf_{\hat{4}}} &= x_1b-x_0c,\\
    {\Pf_{\hat{5}}} &= x_0b-x_1c+x_3d.
\end{align*}
Notice that Pfaffians ${\Pf_{\hat{2}}}$, ${\Pf_{\hat{3}}}$ and ${\Pf_{\hat{4}}}$ may be written as three of the $2 \times 2$ minors of the matrix 
\begin{equation*}
N = 
    \begin{pmatrix}
    x_0 & x_1 & x_2 & 0 \\
    b & c & d & y
    \end{pmatrix},
\end{equation*}
with the other three minors giving the quadrics $Q_0,Q_1,Q_2$. Further, the remaining Pfaffians 
\begin{align*}
    Q_3 &= x_0^2 - x_1^2 + x_2x_3, \\
    F &= bx_0 - cx_1 + dx_3,
\end{align*}
are in the ``rolling factors'' format of Duncan Dicks \cite{reid1989surfaces} with $x_0 \rightarrow b$, $x_1 \rightarrow c$ and $x_2 \rightarrow d$. Moreover, since $d$ is in the ideal $(x_0,x_1,x_2)$ we may write $d=d_0x_0+d_1x_1+d_2x_2$. The quartic $q$ may be obtained from the cubic $F = (d_0x_0+d_1x_1+d_2x_2)x_3-cx_1+bx_0$ by rolling factors and adding an additional term in $y$, which gives
\begin{equation*}
q = (d_0d+d_1b+d_2c)x_3 - c^2 + b^2+Py.
\end{equation*}
\begin{proposition}
Let [441a]ai be the specialisation of the [441a]a family, where $P=yA$ for some quadric $A \in (x_3,x_4,y)$. Then a curve in the [441a]ai family may be deformed to the complete intersection of four quadrics. We show this explicitly by constructing a flat family over $\mathbb{A}^1$.
\end{proposition}
\begin{proof}
Let $t \in \mathbb{A}^1$ be a deformation parameter. We deform the matrix $N$ to $N_{t}$ by replacing $0$ with $t$ in the top right entry:
\begin{equation*}
N_{t}=
    \begin{pmatrix}
    x_0 & x_1 & x_2 & t \\
    b & c & d & y
    \end{pmatrix}.
\end{equation*}
Then the quadrics are now given by
\begin{align*}
    Q_0' &= x_0y - bt, \\
    Q_1' &= x_1y - ct, \\
    Q_2' &= x_2y - dt.
\end{align*}

If $t$ is invertible then the other three minors are not needed to generate the ideal, since
\begin{align*}
    t(x_2b-x_0d) &= x_0Q_2'-x_2Q_0', \\ t(x_2c-x_1d)&=x_1Q_2'-x_2Q_1', \\ t(x_0c-x_1b)&=x_1Q_0'-x_0Q_1'.
\end{align*}
We further deform the remaining Pfaffians to
\begin{align*}
    Q_3' &= x_0^2 - x_1^2 + x_2x_3 + {t}^2A, \\
    F' &= bx_0 - x_1c + x_3d + {t}yA.
\end{align*}

The quartic remains unchanged. Again if $t$ is invertible the cubic Pfaffian is now a tautology since 
\begin{equation*}
    {t}(bx_0 - x_1c + x_3d + {t}yA) = x_1Q_1' - x_0Q_0' - x_3Q_2' + yQ_3'.
\end{equation*}
Similarly, we may write the quartic $q$ as 
\begin{equation*}
    tq =(c-x_3d_1)Q_1'-(b+x_3d_0)Q_0'-x_3d_2Q_2'+yF',
\end{equation*}
so this is also not needed to generate the ideal when $t$ is invertible. Consider the family of curves over $\mathbb{A}^1$ defined by the vanishing of \[I_t=(Q_0',Q_1',Q_2',Q_3',{\Pf_{\hat{2}}}, {\Pf_{\hat{3}}},{\Pf_{\hat{4}}},F',q).\] It follows from the above that the general fibre is a $(2,2,2,2)$ complete intersection defined by the vanishing of 
\begin{align*}
    Q_0' &= x_0y - bt, \\
    Q_1' &= x_1y - ct, \\
    Q_2' &= x_2y - dt, \\
    Q_3' &= x_0^2 - x_1^2 + x_2x_3 + {t}^2A.
\end{align*}
    
On the other hand if $t=0$ the curve defined by $I_t$ is in the [441a]ai family. Since the Hilbert polynomial at each fibre is independent of $t$ this is a flat family over $\mathbb{A}^1$.
\end{proof}

We now consider a different specialisation of the [441a]a family, which admits a deformation to [420]a. We replace the matrix $M$ \eqref{mat} used in the construction of the [441a]a family with the following degenerate version, where we have assumed certain linear dependencies between some entries in the $\text{Jer}_{45}$ matrix:
\begin{equation*}
M' = 
    \begin{pmatrix}
    & b & c & d & 0 \\
    & & x_3 & 0 & x_1 \\
    & & & x_0 & 0 \\
    & & & & x_2 \\
    \end{pmatrix}.
\end{equation*}
The three quadrics $Q_0 = x_0y$, $Q_1 = x_1y$ and $Q_2 = x_2y$ are unchanged. The additional quartic is now given by $bc+x_3cd_0+Ay^2$, and we further set $A=c$, so that $q=bc+x_3cd_0+cy^2$. We refer to this specialisation as [441a]aii. This is a degeneration of [441a]ai which no longer fits into the rolling factors format.
\begin{proposition}
A general curve in the family [441a]aii admits a deformation to a curve of Type~[420]ai, which is a specialisation of [420]a with one extra zero in the Pfaffian matrix.
\end{proposition}
\begin{proof}
The Pfaffians of $M'$ are now
 \begin{gather*}
        {\text{Pf}_{\hat{1}}} = x_0x_1+x_2x_3, \quad
{\text{Pf}_{\hat{2}}} = x_2c, \quad
{\text{Pf}_{\hat{3}}} = x_2b - x_1d, \\
{\text{Pf}_{\hat{4}}} = -x_1c, \quad
{\text{Pf}_{\hat{5}}} = x_0b + x_3d.
\end{gather*}
Let $t \in \mathbb{A}^1$ be a degree 0 deformation parameter. We deform $Q_0$ and  ${\text{Pf}_{\hat{5}}}$ to 
\begin{equation*}
Q_0' = x_0y + tc, \quad {\text{Pf}_{\hat{5}}}' = x_0b + x_3d - tyc.
\end{equation*} 
When $t$ is invertible, we can eliminate ${\text{Pf}_{\hat{2}}}$, ${\text{Pf}_{\hat{4}}}$ and the additional quartic $q$. For example,
\begin{equation*}
tq = (b+d_0x_3)Q_0' -y\text{Pf}_{\hat5}'+d_1x_3Q_1+d_2x_3Q_2.\\
\end{equation*}

The remaining three Pfaffians along with $Q_1$ and $Q_2$ fit into the $4 \times 4$ Pfaffians of the $5 \times 5$ skew-symmetric matrix
\begin{equation}\label{eq:[420]i-matrix}
    \begin{pmatrix}
    & x_2 & x_1 & 0 & 0 \\
    & & tc & d & x_0 \\
    & & & b & -x_3 \\
    & & & & -y \\
    \end{pmatrix}.
\end{equation}
These Pfaffians and the last quadric $Q_0'$ define a curve in the [420]ai family. This is a degeneration of [420]a because the matrix \eqref{eq:[420]i-matrix} has an extra zero entry. Thus the ideal $I_t=(Q_0',Q_1,Q_2,{\text{Pf}_{\hat{1}}},{\text{Pf}_{\hat{2}}},{\text{Pf}_{\hat{3}}},{\text{Pf}_{\hat{4}}},{\text{Pf}_{\hat{5}}}',q)$ defines a curve of Type~[441a]aii when $t=0$, and a curve of Type~[420]ai when $t$ is invertible.
\end{proof}

\subsection{Family~[441b]a}\label{[441b]}
We now discuss the family of curves associated to the stratum $\mathcal{F}_{[441b]}$ with Betti table SSY 6.
\begin{lemma}
    Let $x_0,x_1,y_0,y_1,z_0,z_1$ be coordinates on $\mathbb{P}^5$, and further let $A,B,C,D,P,M,N,Q$ be quadratic forms. Consider the matrices 
     \begin{equation*}
N_1 = 
    \begin{pmatrix}
    x_0 & D' & B' \\
    x_1 & C' & A' \\
    \end{pmatrix}, \quad N_2 = 
    \begin{pmatrix}
    y_0 & Q' & N' \\
    y_1 & P' & M' \\
    \end{pmatrix},
\end{equation*}
where $A',B',C',D'$ are given by the restriction of $A,B,C,D$ to $y_0=y_1=0$ and $M',N',P',Q'$ are given by the restriction of $M,N,P,Q$ to $x_0=x_1=0$. Let $C_1 \subset \mathbb{P}^3_{\left<x_0,x_1,z_0,z_1\right>}$ be defined by the vanishing of the $2 \times 2$ minors of $N_1$ and $C_2 \subset \mathbb{P}^3_{\left<y_0,y_1,z_0,z_1\right>}$ be defined by the vanishing of the $2 \times 2$ minors of $N_2$. Suppose further that the quartics $A'D'-B'C'$ and $P'N'-Q'M'$ agree on $\mathbb{P}^1_{\left<z_0,z_1\right>}$. Then $C=C_1+C_2 \subset \mathbb{P}^5$ is a Gorenstein codimension four curve with Betti table SSY 6, corresponding to the stratum $\mathcal{F}_{[441b]}$. We call this family of curves $[441b]a$.
\end{lemma}
The curves $C_1$ and $C_2$ are both degree 8 and genus 7, and are each residual to a line in a complete intersection. We consider the specialisation [441b]ai where $P=A$, $Q=C$, $M=B$ and $N=D$, so that \[H=AD-BC=PN-QM\] is the quartic in the ideal of $C$. The equations defining $C$ are thus given by the $2 \times 2$ minors of the matrix
\begin{equation*}
N=
    \begin{pmatrix}
    0 & x_0 & x_1 \\
    y_0 & D & C \\
    y_1 & B & A \\
    \end{pmatrix}.
\end{equation*}

\begin{proposition}
Any curve in the [441b]ai family may be deformed to the complete intersection of four quadrics. We explicitly construct a flat family over $\mathbb{A}^1$ whose general fibre is in the [400]a family and whose special fibre is in the [441b]ai family.
\end{proposition}
\begin{proof}
Introduce a degree 0 deformation parameter $t \in \mathbb{A}^1$ and consider the deformed matrix
\begin{equation*}
N_t=
    \begin{pmatrix}
    -t & x_0 & x_1 \\
    y_0 & D & C \\
    y_1 & B & A \\
    \end{pmatrix}.
\end{equation*}
The $2 \times 2$ Pfaffians of this matrix define a $(2,2,2,2)$ complete intersection. The quadrics become
\begin{align*}
    Q_1' &= x_0y_0 + tD, \\
    Q_2' &= x_0y_1 + tB, \\
    Q_3' &= x_1y_0 + tC, \\
    Q_4' &= x_1y_1 + tA.
\end{align*}
    
If $t$ is invertible then the four cubics and quartic are tautologies. Indeed, we have
\begin{align*}
    tF_1 &= x_0Q_4' - x_1Q_2', \\
    tF_2 &= x_0Q_3' - x_1Q_1', \\
    tG_1 &= y_0Q_4' - y_1Q_3', \\
    tG_2 &= y_0Q_2' - y_1Q_1', \\
    tH &= DQ_4' - BQ_3' + x_1G_2.
\end{align*}

We see that $(Q_1',Q_2',Q_3',Q_4')$ defines a complete intersection. We again consider the family of curves defined by ideals of the form 
\begin{equation*}
I_t=(Q_1',Q_2',Q_3',Q_4',F_1,F_2,G_1,G_2,H).
\end{equation*}
As shown above, the general fibre is the complete intersection of the four quadrics $(Q_1',Q_2',Q_3',Q_4')$, whereas the special fibre is in the [441b]ai family.
\end{proof}

We now exhibit a specialisation [441b]aii of the [441b]ai family which admits a deformation to [430]a. Let [441b]aii be the specialisation of [441b]ai with $A=x_0a+y_0d$ and $C=x_0c+y_1b$, so that the matrix $N$ becomes:
\begin{equation*}
    N=
    \begin{pmatrix}
    0 & x_0 & x_1 \\
    y_0 & D & x_0c + y_1b \\
    y_1 & B & x_0a + y_0d \\
    \end{pmatrix}.
\end{equation*}

\begin{proposition}
There is a flat family over $\mathbb{A}^1$ whose special fibre at $t=0$ is in the [441b]aii family and whose general fibre when $t$ is invertible is in the [430]a family.
\end{proposition}
\begin{proof}
The $2\times2$ minors of $N$ can be simplified using the quadrics $Q_i$ to give the following equations for the fibre over $t=0$:
\begin{gather*}
Q_1=x_1y_1,\quad Q_2=x_1y_0,\quad Q_3=x_0y_0,\quad Q_4=x_0y_1, \\
F_1=Bx_1-ax_0^2,\quad F_2=Dx_1-cx_0^2,\quad F_3=Dy_1-By_0, \quad F_4=y_0^2d-y_1^2b, \\
H=D(x_0a+y_0d)-B(x_0c+y_1b).
\end{gather*}
We deform two of the quadrics and three of the cubics to:
\begin{gather*}
  Q_3' = x_0y_0 - tby_1, \quad Q_4' = x_0y_1 - tdy_0, \quad F_1' = Bx_1 - a(x_0^2 - t^2bd), \\
     F_2' = Dx_1 - c(x_0^2 - t^2bd), \quad F_3' = D(y_1+ta) - B(y_0+tc).
\end{gather*}
Then we have 
\begin{equation*}
tF_4 = y_1Q_3' -y_0Q_4',\quad tH = -x_0F_3' -BQ_3'+DQ_4',
\end{equation*}
so that when $t$ is invertible, $F_4$ and $H$ are not needed as ideal generators.

We show that the remaining seven ideal generators fit into family [430]a. Recall the Cramer's rule format for the [430]a family, defined using a matrix $M$, vector $v$ and parameter $s$. Choosing $M,v,s$ to be 
\begin{equation*}
M = 
    \begin{pmatrix}
    \tfrac{1}{t}D & -\tfrac{1}{t}B & a & c \\
    -tb & x_0 & 0 & 0 \\
    x_0 & -td & 0 & 0
    \end{pmatrix}, \quad
v = 
    \begin{pmatrix}
    y_1 \\
    y_0 \\
    D \\
    -B
    \end{pmatrix}, \quad
    s = x_1.
\end{equation*}
expresses the seven remaining ideal generators as an element of [430]a
Note that the quadratic forms $D,B$ appear both in $M$ and $v$. 

In conclusion the ideal $I_t=(Q_1,Q_2,Q_3',Q_4',F_1',F_2',F_3',F_4,H)$ defines a family of curves $C_t$, where the special fibre at $t=0$ is in the [441b]aii family, and is in the [430]a family when $t$ is invertible.
\end{proof}

The degeneration from [441b]ai to [441b]aii breaks the two degree 8 components $C_i$ of a curve in the [441b]ai family into $C_i'+\PP^1_{\left<z_0,z_1\right>}$. Thus a curve of Type~[441b]aii contains the residual line $\PP^1_{\left<z_0,z_1\right>}$ as a nonreduced component. The above deformation to [430]a smooths this double line into the residual conic in the construction of the general curve of Type~[430]a curve. We do not know if there is a reduced curve of Type~[441b]ai which deforms to [430]a.

\section{Deformations in degree 17}\label{sec4}
We describe three families in degree 17, along with their specialisations. The results for this section can be seen in figure \ref{fig:deg17}.

\begin{figure}[ht]
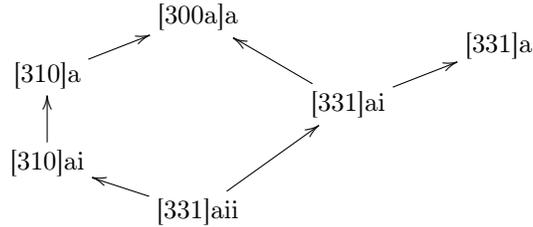

\[\xygraph{
!{<0cm,0cm>;<1cm,0cm>:<0cm,1.3cm>::}
!{(0,5)}*+{{\text{[300a]a}}}="1"
!{(-2,4.4)}*+{{\text{[310]a}}}="2"
!{(4,4.7)}*+{{\text{[331]a}}}="3"
!{(2,4.1)}*+{{\text{[331]ai}}}="4"
!{(0,3)}*+{{\text{[331]aii}}}="5"
!{(-2,3.5)}*+{{\text{[310]ai}}}="6"
"2":"1"
"4":"1"
"5":"6"
"4":"3"
"5":"4"
"6":"2"
}\]\caption{Degree 17 strata and incidences}\label{fig:deg17}
\end{figure}


\subsection{Family~[300a]a}
Family~[300a]a is one of the nonsingular families of \cite{coughlan2016arithmetically} (see also \cite[\S 2.8]{Papadakis2000Kustin--MillerComplexes}).
\begin{lemma}
    Let $M$ be a $3 \times 4$ matrix and $v$ a $4 \times 1$ column vector with linear entries, with a further degree 2 parameter $s$. Then the variety defined by the equations 
    \[Mv=0, \quad \bigwedge^3 M=sv\]
    has Betti table CGKK 4. This is family [300a]a, and it is a lift of the stratum $\mathcal{F}_{[300a]}$ in the space of quartics in four variables.
\end{lemma}

\begin{table}[h!]
    \begin{equation*}\begin{array}{c|c c c c c}
        & 0 & 1 & 2 & 3 & 4 \\ \hline
       0 & 1 & - & - & - & - \\
      1 & - & 3 & - & - &-\\
      2 &- & 4 & 12 & 4 & -\\
      3 &- & - & - & 3 & - \\
      4 & - & - & - & - & 1
    \end{array}
\end{equation*}
\caption*{Betti table CGKK 4~\cite{schenck2020calabiyau}}
\label{tab:table8}
\end{table}

\subsection{Family~[310]a}\label{sec!type-[310]}
We next describe family~[310]a.
\begin{table}[h!]
    \begin{equation*}\begin{array}{c|c c c c c}
        & 0 & 1 & 2 & 3 & 4 \\ \hline
       0 & 1 & - & - & - & - \\
      1 & - & 3 & 1 & - &-\\
      2 &- & 5 & 12 & 5 & -\\
      3 &- & - & 1 & 3 & - \\
      4 & - & - & - & - & 1
    \end{array}
\end{equation*}
\caption*{Betti table SSY 2~\cite{schenck2020calabiyau}}
\label{tab:table9}
\end{table}
\begin{lemma}
    Let $x_0,\dots,x_7,y$ be coordinates on $\mathbb{P}^8$ and let $M$ be the $5 \times 5$ skew-symmetric matrix
    \begin{equation*}
    M=
    \begin{pmatrix}
    & A & B & C & D \\
    & & x_2 & x_3 & x_4 \\
    & & & x_5 & x_6 \\
    & & & & x_7
    \end{pmatrix}
\end{equation*}
where the entries of the first row are
    \[A=a_0x_0+a_1x_1,\quad B=b_0x_0+b_1x_1, \quad C=c_0x_0+c_1x_1, \quad D=d_0x_0+d_1x_1,\]
with $a_i,b_i,c_i,d_i$ linear forms.
The $4 \times 4$ Pfaffians of $M$ are four cubics and one quadric, $\Pf_{\hat{1}}$. Let
\[Q_0=x_0y, \quad Q_1=x_1y, \quad Q_2=\Pf_{\hat{1}}+yL,\]
where $L$ is a linear form. Next define the cubic 
\begin{multline}\label{eq![310]-cubic}
H=(c_0d_1-c_1d_0)x_2-(b_0d_1-b_1d_0)x_3+(a_0d_1-a_1d_0)x_4 \\
+(b_0c_1-b_1c_0)x_5-(a_0c_1-a_1c_0)x_6+(a_0b_1-a_1b_0)x_7,
\end{multline}
and let $E$ be a general quadratic form.
Then the cubic Pfaffians of $M$, along with $Q_0,Q_1,Q_2$ and $H+yE$ define a fourfold $X\subset\mathbb{P}^8$ with Betti table SSY 2. Intersecting $X$ with a $5$-dimensional linear subspace gives a family of curves in $\PP^5$ which we call [310]a.
\end{lemma}

 Since the quadrics $Q_0$ and $Q_1$ are reducible, we observe that $X=X_1+X_2$, with $X_1 \subset \mathbb{P}^7_{\left<x_0,\dots,x_7\right>}$ and $X_2 \subset \mathbb{P}^6_{\left<x_2,\dots,x_7,y\right>}$. Let $\Gamma$ be the variety defined by the Pfaffians of $M$ in $\mathbb{P}^7$, and $Y$ the variety cut out by $(x_0,x_1,\Pf_{\hat{1}})$. Note that $Y$ is a component of $\Gamma$, and the Pfaffians of $M$ together with the cubic $H$ generate the colon ideal $(I_{\Gamma} \colon I_{Y})$. Thus $X_1$ is residual to $Y$ in $\Gamma$. On the other hand, $X_2$ is cut out by the ideal $(Q_2,H+yE)$ in $\mathbb{P}^6$, and is thus a $(2,3)$ complete intersection. Intersecting this construction with a general $\mathbb{P}^5$ produces a Gorenstein curve of codimension four, as in \cite{ablett2021halfcanonical}.

\begin{proposition}
There is a flat family over $\mathbb{A}^1$ whose special
fibre is a variety of Type [310]a and whose general fibre is a variety of Type [300a]a.
\end{proposition}
\begin{proof}
 First we collect the terms of $ \Pf_{\hat{5}}$ involving $x_0$, $x_1$ as follows:
 \begin{equation*}
 \Pf_{\hat{5}} = (x_5a_0 - x_3b_0 + x_2c_0)x_0 + (x_5a_1 - x_3b_1 + x_2c_1)x_1.
 \end{equation*}
 Let $t \in \mathbb{A}^1$ be a deformation parameter. We deform the quadrics $Q_0$ and $Q_1$ to 
 \begin{equation*}
    Q_0' = x_0y + t(x_5a_1 - x_3b_1 + x_2c_1), \quad Q_1' = x_1y + t(x_5a_0 - x_3b_0 + x_2c_0).
\end{equation*}
Then $t\Pf_{\hat{5}} = x_1Q_0' - x_0Q_1'$. Hence if $t$ is invertible, $\Pf_{\hat{5}}$ is no longer needed as an ideal generator. The remaining cubic Pfaffians deform to
\begin{align*}
\Pf_{\hat{2}}' &= {\Pf_{\hat{2}}} + t(x_5E + (b_1c_0 - b_0c_1)L), \\
\Pf_{\hat{3}}' &= {\Pf_{\hat{3}}} + t(x_3E + (a_1c_0 - a_0c_1)L), \\
\Pf_{\hat{4}}' &= {\Pf_{\hat{4}}} + t(x_2E + (a_1b_0 - a_0b_1)L).
\end{align*}
When $t$ is invertible we may also replace the ideal generator $H+yE$ with the polynomial
\[(H+yE)-\frac1t(d_0Q_0'+d_1Q_1').\]
Hence when $t$ is invertible, the remaining three quadrics and four cubic generators fit into the ``Cramer's rule'' format in the following way:
\begin{equation*}
N = 
    \begin{pmatrix}
    x_1 & a_0 & - b_0 & c_0 \\
    x_0 & -a_1 & b_1 & -c_1 \\
    tL & x_4 & -x_6 & x_7
    \end{pmatrix},
\quad v=
    \begin{pmatrix}
    \frac{1}{t}y \\
    x_5 \\
    x_3 \\
    x_2
    \end{pmatrix},
    \quad
    s = D + tE.
\end{equation*}
For invertible $t$, the three quadrics are obtained from $Nv=0$ and the remaining four cubics are given by $\bigwedge^3 N=sv$.

It follows that  $(Q_0',Q_1',Q_2,\Pf_{\hat{2}}',\Pf_{\hat{3}}',\Pf_{\hat{4}}',\Pf_{\hat{5}},H+yE)$ defines a family of codimension four varieties whose general fibre is of Type [300a]a, and whose special fibre at $t=0$ is of Type~[310]a. Since the Hilbert polynomial at every fibre does not depend on $t$, this family is flat.
\end{proof}

\subsection{Family~[331]a}
We now describe a family of varieties with Betti table SSY 5.
\renewcommand{\arraystretch}{1}
\begin{table}[h!]
    \begin{equation*}\begin{array}{c|c c c c c}
        & 0 & 1 & 2 & 3 & 4 \\ \hline
       0 & 1 & - & - & - & - \\
      1 & - & 3 & 3 & 1 &-\\
      2 &- & 7 & 14 & 7 & -\\
      3 &- & 1 & 3 & 3 & - \\
      4 & - & - & - & - & 1
    \end{array}
\end{equation*}
\caption*{Betti table SSY 5~\cite{schenck2020calabiyau}}
\label{tab:table10}
\end{table}
\begin{lemma}
Let $x_0,\dots,x_{12},y$ be coordinates on $\mathbb{P}^{13}$ and let $M$ be the skew-symmetric matrix
    \begin{equation}\label{eq!matrix-[331]}
    M=
    \begin{pmatrix}
    & x_3 & x_4 & x_6 & x_9 & a_0x_0 & 0 \\
    & & x_5 & x_7 & x_{10} & b_1x_1 & 0 \\
    & & & x_8 & x_{11} & 0 & 0 \\
    & & & & x_{12} & 0 & x_0 \\
    & & & & & 0 & x_1 \\
    & & & & & & x_2 \\
    \end{pmatrix}
\end{equation}
where $a_0,b_1$ are scalars.
Let $X\subset\PP^{13}$ be defined by the seven $6\times6$ Pfaffians of $M$, the quadrics
\[Q_0=x_0y, \quad Q_1=x_1y, \quad Q_2=x_2y\]
and the quartic $H=q+yF$ where $q$ is the following combination of the $4\times4$ Pfaffians of $M$:
\begin{equation*}
q={a_0\Pf_{1234}\Pf_{2345}}-{b_1\Pf_{1345}\Pf_{1235}},
\end{equation*}
and $F$ is a general cubic form.

 Then $X$ is a $9$-dimensional Gorenstein codimension four variety with Betti table SSY 5. The intersection of $X$ with a general linear subspace of dimension $5$ is a curve $C$ and we call this family [331]a.
\end{lemma}
In the language of Brown, Kerber and Reid \cite{brown2012fano} $M$ is a $\text{Jer}_{67}$, with all the $6 \times 6$ cubic Pfaffians of $M$ lying in the ideal $(x_0,x_1,x_2)$. Indeed $M$ is a normal form for the following $\text{Jer}_{67}$ matrix:
\[M'=\begin{pmatrix}
    & x_3 & x_4 & x_6 & x_9 & A & S \\
    & & x_5 & x_7 & x_{10} & B & T \\
    & & & x_8 & x_{11} & C & U \\
    & & & & x_{12} & D & x_0 \\
    & & & & & E & x_1 \\
    & & & & & & x_2 \\
    \end{pmatrix}\]
    where $A,B,C,D,E,S,T,U$ are general in $(x_0,x_1,x_2)$.

In the statement of the lemma, we write $\Pf_{ijkl}$ for the $4 \times 4$ Pfaffian of the skew-symmetric submatrix of $M$ which is obtained by removing those row-columns whose index does not appear in $\{i,j,k,l\}$. Hence,
\begin{align*}
\Pf_{1234} &= x_3x_8 - x_4x_7 + x_5x_6, &
\Pf_{1235} &= x_3x_{11} - x_4x_{10} + x_5x_9, \\
\Pf_{1236} &= b_1x_1x_4 - a_0x_0x_5, &
\Pf_{2345} &= x_5x_{12} - x_7x_{11} + x_8x_{10}, \\
\Pf_{1345} &= x_4x_{12} - x_6x_{11} + x_8x_9, &
\Pf_{1245} &= x_3x_{12} - x_6x_{10} + x_7x_9.
\end{align*}
Then three of the $6\times6$ Pfaffians are
\begin{equation*}
\begin{split}
& \Pf_{\hat{6}} = x_1{\Pf_{1234}} - x_0{\Pf_{1235}}, \\
& \Pf_{\hat{5}} = x_0{\Pf_{1236}} + x_2{\Pf_{1234}}, \\
& \Pf_{\hat{4}} = x_1{\Pf_{1236}} + x_2{\Pf_{1235}}.
\end{split}
\end{equation*}

Since the quadric generators are reducible, $X=X_1+X_2$, where $X_1 \subset \mathbb{P}^{12}_{\left<x_0,\dots,x_{12}\right>}$ and $X_2 \subset \mathbb{P}^{10}_{\left<x_3\dots,x_{12},y\right>}$. Further, let $\Gamma \subset \mathbb{P}^{12}$ be the variety cut out by the seven cubic Pfaffians of $M$, and $l$ be the linear subspace given by the vanishing of $(x_0,x_1,x_2)$. Then the quartic $q$ and seven cubic Pfaffians generate the colon ideal $(I_{\Gamma} \colon I_l)$. Thus $X_1$ is residual to $l$ in $\Gamma$. On the other hand, $X_2 \subset \mathbb{P}^{10}$ is a quartic hypersurface cut out by the polynomial $H=q+yF$.

\begin{proposition}
Let [331]ai denote the special case of [331]a where the cubic $F$ is zero so that $H=q$. There is a flat family over $\mathbb{A}^1$ whose special fibre is a variety of Type~[331]ai and general fibre is a variety of Type~[300a]a.
\end{proposition}

\begin{proof}
Let $t \in \mathbb{A}^1$ be a deformation parameter. We first deform the quadrics using $\Pf_{1234}$, $\Pf_{1235}$ and ${\Pf_{1236}}$:

\begin{equation*}
    Q_0' = x_0y + t\Pf_{1234}, \quad Q_1' = x_1y + t\Pf_{1235}, \quad Q_2' = x_2y - t\Pf_{1236}.
\end{equation*}
If $t$ is invertible, then
\begin{equation*}
    t\Pf_{\hat{6}}=x_1Q_0' - x_0Q_1', \quad 
     t\Pf_{\hat{5}} = x_2Q_0' - x_0Q_2', \quad
     t\Pf_{\hat{4}} = x_2Q_1' - x_1Q_2'.
\end{equation*}
It follows that $\Pf_{\hat{4}}, \Pf_{\hat{5}}$ and $\Pf_{\hat{6}}$ are redundant as ideal generators when $t$ is invertible. Moreover,
\[tH=a_0\Pf_{2345}Q_0'-b_1\Pf_{1345}Q_1'-y\Pf_{\hat{7}},\]
and consequently if $t$ is invertible $H$ is also redundant. The remaining four cubic Pfaffians and three quadrics fit into the ``Cramer's rule'' format as follows:
\begin{equation*}
N = 
    \begin{pmatrix}
    x_0 & x_8 & -x_7 & x_6 \\
    x_1 & x_{11} & -x_{10} & x_9 \\
    x_2 & 0 & -b_1x_1 & a_0x_0 \\
    \end{pmatrix},\quad
v = 
    \begin{pmatrix}
    \frac{1}{t}y \\
    x_3 \\
    x_4 \\
    x_5 
    \end{pmatrix}, \quad
    s=- x_2x_{12}.
\end{equation*}
Then the three quadrics $Q_i'$ are given by $Nv = 0$, and the remaining four cubics $\Pf_{\hat{4}}, \Pf_{\hat{5}}$, $\Pf_{\hat{6}}$, $\Pf_{\hat{7}}$ come from $\bigwedge^3 N=sv$.

Thus, the ideal $I=(Q_0',Q_1',Q_2',\Pf_{\hat{1}},\dots,\Pf_{\hat{7}},H)$ defines a flat family of varieties, whose special fibre at $t=0$ is of Type [331]ai and whose general fibre is of Type [300a]a.
\end{proof}

Next we study the varieties which lie between the families [310]a and [331]a. We consider the degeneration of [331]ai to [331]aii, which is obtained by setting the $x_0$ entry in the last column of \eqref{eq!matrix-[331]} to zero and again setting the cubic $F=0$ in $H$. Thus the matrix $M$ becomes
\begin{equation*}
M=
    \begin{pmatrix}
    & x_3 & x_4 & x_6 & x_9 & a_0x_0 & 0 \\
    & & x_5 & x_7 & x_{10} & b_1x_1 & 0 \\
    & & & x_8 & x_{11} & 0 & 0 \\
    & & & & x_{12} & 0 & 0 \\
    & & & & & 0 & x_1 \\
    & & & & & & x_2 \\
    \end{pmatrix}.
\end{equation*}
The construction proceeds as before, with $X$ being defined by the quadrics
\begin{equation*}
Q_0=x_0y, \quad Q_1 = x_1y, \quad Q_2 = x_2y,
\end{equation*} the 7 cubic Pfaffians of $M$
\begin{align*}
\Pf_{\hat{1}} &= {x_2\Pf_{2345}} - b_1x_1^2x_8, \\
\Pf_{\hat{2}} &= {x_2\Pf_{1345}} - a_0x_0x_1x_8, \\
\Pf_{\hat{3}} &= {x_2\Pf_{1245}}  - a_0x_0x_1x_7 + b_1x_1^2x_6, \\
\Pf_{\hat{4}} &= {x_2\Pf_{1235}} - a_0x_0x_1x_5 + b_1x_1^2x_4, \\
\Pf_{\hat{5}} &= x_2\Pf_{1234}, \\
\Pf_{\hat{6}} &= x_1\Pf_{1234}, \\
\Pf_{\hat{7}} &= {a_0x_0\Pf_{2345}}-{b_1x_1\Pf_{1345}},
\end{align*}
and the additional quartic
\begin{equation*}
H=\Pf_{1234}\Pf_{2345},
\end{equation*}
where all $\Pf_{ijkl}$ are as before.
\begin{proposition}
There exists a flat deformation whose central fibre is of Type~[331]aii and whose general fibre is of Type~[310]ai (see the proof for a description of family~[310]ai).
\end{proposition}
\begin{proof}
Introducing a deformation parameter $t \in \mathbb{A}^1$, we first deform $Q_0$ to $Q_0'=x_0y+{t\Pf_{1234}}$. Thus Pfaffians $\Pf_{\hat{5}}$ and $\Pf_{\hat{6}}$ are not required as ideal generators when $t$ is invertible and  
\begin{equation*}
ta_0H = a_0\Pf_{2345}Q_0'-b_1\Pf_{1345}Q_1-y\Pf_{\hat{7}}
\end{equation*}
as before, so that $H$ is also not required.

Moreover, Pfaffians $\Pf_{\hat{1}},\Pf_{\hat{2}},\Pf_{\hat{3}},\Pf_{\hat{4}}$ and the quadric $\Pf_{1234}$ are the $4 \times 4$ Pfaffians of the $5 \times 5$ matrix 
\begin{equation*}
N=
    \begin{pmatrix}
    & x_2x_9-a_0x_0x_1 & x_2x_{10}-b_1x_1^2 & x_2x_{11} & x_2x_{12}  \\
    & & x_3 & x_4 & x_6 \\
    & & & x_5 & x_7 \\
    & & & & x_{8} \\
    \end{pmatrix}.
\end{equation*}
Comparing this with \S\ref{sec!type-[310]}, with  $Q_1$, $Q_2$, $Q_0'$ as the three quadrics, and $\Pf_{\hat{7}}$ as the extra cubic \eqref{eq![310]-cubic}, we see that the general fibre is in the [310]a family. Since $N$ is quite special, we call the resulting subfamily [310]ai.


Thus the ideal $I_t = (Q_0',Q_1,Q_2,\text{Pf}_{\hat{1}},\text{Pf}_{\hat{2}},\text{Pf}_{\hat{3}},\text{Pf}_{\hat{4}},\text{Pf}_{\hat{5}},\text{Pf}_{\hat{6}},\text{Pf}_{\hat{7}},H)$ defines a variety $X_0$ of Type~[331]aii when $t=0$ and a variety $X_{t}$ of Type~[310]ai when $t$ is invertible. Dimension is preserved, so this is a flat family over $\mathbb{A}^1$.
\end{proof}

Let us call the above deformation $\mathcal{X}\to\mathbb{A}^1$. Then the total space $\mathcal{X}$ has three components, each of which dominates $\mathbb{A}^1$ by flatness. The first component $\mathcal{X}_1\to\mathbb{A}^1$ is a trivial deformation whose fibre $X_1$ is the degree $11$ component of a general variety in the [310]ai family. The second component $\mathcal{X}_2\to\mathbb{A}^1$ is again trivial, and its fibre $X_2$ is the quadric hypersurface $(\Pf_{1234}=0)$ in the linear subspace $x_0=x_1=x_2=0$. The last component $\mathcal{X}_3\to\mathbb{A}^1$ is a nontrivial deformation, defined by $(x_0y-t\Pf_{1234}=\Pf_{2345}=0)$ in the linear subspace $x_1=x_2=0$. The general fibre $X_3$ of $\mathcal{X}_3$ is a complete intersection of two quadrics and $X_2+X_3$ is the $(2,3)$ complete intersection $(x_0y-t\Pf_{1234}=x_0\Pf_{2345}=0)$. The central fibre $Y_3=Y_3'+Y_3''$ breaks into two hyperplane sections $Y_3'\colon(y=0)$ respectively $Y_3''\colon(x_0=0)$ of the quadric $(\Pf_{2345}=0)$. 

Returning to the original deformation $\mathcal{X}\to\mathbb{A}^1$, we see that the general fibre $\mathcal{X}_t=X_1+(X_2+X_3)$ is in a degeneration of the [310]a family, which we called [310]ai. Moreover, $\mathcal{X}_0=(X_1+Y_3')+(X_2+Y_3'')$ where $(X_1+Y_3')$ is the degenerate degree $13$ component of a general variety of Type~[331]aii, and $(X_2+Y_3'')$ is a degenerate quartic ear.

\section{Deformations in degree 18}\label{sec5}
In degree 18, we discuss the singular family [210]a with Betti table [210]. We show that a variety in the specialisation [210]ai admits a smoothing to a variety in the [200]a family described in \cite{coughlan2016arithmetically} involving a linear section of the fourfold $\mathbb{P}^2 \times \mathbb{P}^2$.
\subsection{Families~[200]a and [200]b}
Two families [200]a and [200]b are outlined in \cite{coughlan2016arithmetically}. These higher degree varieties are more complicated than those previously outlined, so their descriptions are not as complete. One of the families from \cite{coughlan2016arithmetically} is bilinked to a linear section of $\PP^2\times\PP^2$ in a $4$-dimensional $(2,2,3)$ complete intersection. We call this family [200]a. The other family is bilinked to $\PP^1\times\PP^1\times\PP^1$ and we call this [200]b. 
Both constructions have the CGKK 7/8 Betti table. For curves, there is no difference between [200]a and [200]b. Both families are bilinked to the normal elliptic curve of degree $6$. Thus we may refer to this family of curves as [200].

\renewcommand{\arraystretch}{1}
\begin{table}[h!]
    \begin{equation*}\begin{array}{c|c c c c c}
        & 0 & 1 & 2 & 3 & 4 \\ \hline
       0 & 1 & - & - & - & - \\
      1 & - & 2 & - & - &-\\
      2 &- & 8 & 18 & 8 & -\\
      3 &- & - & - & 2 & - \\
      4 & - & - & - & - & 1
    \end{array}
\end{equation*}
\caption*{Betti table CGKK 7/8~\cite{schenck2020calabiyau}}
\label{tab:table11}
\end{table}

\subsection{Family~[210]a}
The curve $C$ of Type~[210]a was first described in \cite{ablett2021halfcanonical}.
\begin{lemma}
Let $x_0,\dots,x_5,y_0,y_1,z_0,z_1,z_2,w$ be coordinates on $\PP^{11}$ and let $M$ be the skew-symmetric matrix
\begin{equation*}
M = 
    \begin{pmatrix}
    & x_0 & x_1 & x_2 & y_0 & y_1 & 0 \\
    & & x_3 & x_4 & y_1 & ay_0 & 0 \\
    & & & x_5 & 0 & by_1 & y_0 \\
    & & & & 0 & y_0 & y_1 \\
    & & & & & z_0 & z_1 \\
    & & & & & & z_2
    \end{pmatrix}
\end{equation*}
where $a,b$ are scalars.
We define quadrics
\begin{equation*}
    Q_0 = y_0w, \quad Q_1 = y_1w, \quad Q_2=x_0x_5-x_1x_4+x_2x_3,
\end{equation*}
where $Q_2$ is the Pfaffian of the upper left $4 \times 4$ block of $M$. Finally, let $F_0$ be the following cubic:
\begin{align*}
F_0=ab(y_0y_1z_1 &- x_2z_1^2) + a(y_0^2z_0 - x_1z_0z_1 + x_5z_1^2) +
b(y_1^2z_2 + x_0z_1^2 - x_4z_1z_2) \\
  &- x_0z_0^2+ x_4z_0z_1 + x_3z_1^2 + x_2z_0z_2 - x_3z_0z_2 + x_1z_1z_2 - x_5z_2^2 \\
   &- y_0^2z_2 - y_0y_1z_1 - y_1^2z_0.
\end{align*}
Suppose that $X$ be defined by the quadrics $Q_0,Q_1$, the seven cubic $6\times6$ Pfaffians of $M$, the cubic $F_1=F_0+wG$ where $G$ is a quadratic form, and the further cubic $F_2=wQ_2$.

Then $X$ is a Gorenstein codimension $4$ variety with Betti table SSY 1. The intersection of $X$ with a $5$-dimensional linear subspace is a curve $C$ and we call this family of curves [210]a.
\end{lemma}
\renewcommand{\arraystretch}{1}
\begin{table}[h!]
    \begin{equation*}\begin{array}{c|c c c c c}
        & 0 & 1 & 2 & 3 & 4 \\ \hline
       0 & 1 & - & - & - & - \\
      1 & - & 2 & 1 & - &-\\
      2 &- & 9 & 18 & 9 & -\\
      3 &- & - & 1 & 2 & - \\
      4 & - & - & - & - & 1
    \end{array}
\end{equation*}
\caption*{Betti table SSY 1~\cite{schenck2020calabiyau}}
\label{tab:table14}
\end{table}

We start from the following skew-symmetric matrix $M'$
\begin{equation*}
M' = 
    \begin{pmatrix}
    & x_0 & x_1 & x_2 & a_{15} & a_{16} & a_{17} \\
    & & x_3 & x_4 & a_{25} & a_{26} & a_{27} \\
    & & & x_5 & a_{35} & a_{36} & a_{37} \\
    & & & & a_{45} & a_{46} & a_{47} \\
    & & & & & z_0 & z_1 \\
    & & & & & & z_2
    \end{pmatrix}
\end{equation*}
where the $a_{ij}$ are generic linear combinations of $y_0$ and $y_1$ so that the $6\times 6$ cubic Pfaffians of $M'$ lie in the ideal $(y_0,y_1,Q_2)$.
Since we assumed that $M'$ is generic, by performing row-column operations, rescaling and relabelling variables in a prudent manner, we may reduce $M'$ to the normal form $M$ displayed in the lemma.

Since the quadrics $Q_0,Q_1$ are reducible, it follows that $X = X_1 + X_2$ where $X_1$ lies in the copy of $\PP^{10}$ defined by $w=0$, while $X_2$ lies in $\PP^{9}$ defined by $y_0=y_1=0$. 

By construction, $X_1$ is residual to the quadric $(y_0,y_1,Q_2)$ in the variety $\Gamma$ defined by the $6\times6$ Pfaffians of $M$. Thus $X_1\in\PP^{10}$ is defined by the seven cubic Pfaffians, along with the cubic $F_0$, which we obtained by computing the colon ideal $(I_{\Gamma}:(y_0,y_1,Q_2))$. 

The other component $X_2$ is the $(2,3)$ complete intersection defined by $Q_2$ and $F_1$. In \cite{ablett2021halfcanonical}, the quadric $Q_2$ had to be of rank $2$, but here we can have quadrics of rank $6$.

To construct a smoothing to Type~[200], we consider the specialisation to family [210]ai where we assume that $G=0$.
\begin{proposition}
A curve of Type~[210]ai can be deformed to a curve of Type~[200].
\end{proposition}
\begin{proof}
We now construct a flat family over $\mathbb{A}^1$, introducing $t \in \mathbb{A}^1$ as our deformation parameter. Note that the final four Pfaffians, ${\Pf_{\hat{4}}},\dots,{\Pf_{\hat{7}}}$ are all in the ideal $(y_0,y_1)$, so each Pfaffian may therefore be written as $A_iy_0 + B_iy_1$ for appropriate quadrics $A_i$, $B_i$. 

For instance, ${\Pf_{\hat{4}}} = Ay_0 + By_1$ where 
\begin{align}
A&=-ay_0^2+ax_1z_1+x_0z_0+x_3z_2,\\
B&=-bx_0z_1+y_0y_1-x_3z_1-x_1z_2.
\end{align}
If we deform the quadrics to $Q_0' = y_0w + tB$, $Q_1' = y_1w - tA$, then
\begin{equation*}
    t{\Pf_{\hat{4}}} = y_0Q_1' - y_1Q_0'.
\end{equation*}
Thus, when $t$ is invertible $\Pf_{\widehat{4}}$ becomes redundant. We further deform the cubic $F_2=wQ_2$ to 
\begin{multline}
F_2' = wQ_2 - t(ax_1x_2y_0-ax_1^2y_1-bx_0x_4y_0+bx_0x_2y_1+bx_0x_3y_1\\
-x_0x_1y_0-x_3x_4y_0+x_3^2y_1-x_0x_5y_1).
\end{multline}
This extra term is chosen so that the syzygies $y_iF_2= Q_2Q_i$ for $i=0,1$ extend to the deformed ideal. For completeness, the extensions are:
\begin{align*}
y_0F_2' &= Q_2Q_0'+t((bx_0+x_3)\Pf_{\hat{6}}+x_1\Pf_{\hat{7}}), \\
y_1F_2' &= Q_2Q_1'+t(ax_1\Pf_{\hat{6}}+x_0\Pf_{\hat{1}}+x_3\Pf_{\hat{7}}).
\end{align*}

Now consider the ideal 
\begin{equation*}
I_t = (Q_1',Q_2',F_1,F_2',{\text{Pf}_{\hat{1}}},\dots,{\text{Pf}_{\hat{7}}}),
\end{equation*}
defining a variety $X_t$. The fibre over $t=0$ is of Type~[210]i. Moreover, the general fibre is irreducible, reduced, and its ideal is generated by eight cubics and two quadrics, and has Betti table CGKK 7/8. 

Let $C_t$ be a general fibre of the family. We use Magma \cite{MR1484478} to work out the bilinkage of $C_t$ in a $(2,2,3)$ complete intersection. We find that 
$C_t$ is indeed bilinked to a normal elliptic curve of degree $6$. We therefore have a one parameter deformation of curves with general fibre of Type~[200] and special fibre at $t=0$ of Type~[210]ai.
\end{proof}


\paragraph{Acknowledgements}
We thank Miles Reid for his invaluable input throughout this project, as well as Jan Stevens, whose original work on the degree 15 case was a key inspiration for the rest of the deformations in this paper, and who commented on a previous version. We also thank Diane Maclagan for her substantial help in editing this paper. Finally, we thank the referees for their helpful comments which improved the paper. Ablett was funded through the Warwick Mathematics Institute Centre for
Doctoral Training, with support from the University of Warwick and the UK Engineering and Physical Sciences Research Council (EPSRC grant EP/W523793/1).

\end{document}